\documentclass[12pt, reqno]{amsart}
 \usepackage{amsmath, amsthm, amscd, amsfonts, amssymb, graphicx, color, float}
\usepackage[bookmarksnumbered, colorlinks, plainpages]{hyperref}

\newtheorem{theorem}{Theorem}[section]
\newtheorem{lemma}[theorem]{Lemma}
\newtheorem{proposition}[theorem]{Proposition}
\newtheorem{corollary}[theorem]{Corollary}
\theoremstyle{definition}
\newtheorem{definition}[theorem]{Definition}

\theoremstyle{remark}
\newtheorem{remark}[theorem]{Remark}
\numberwithin{equation}{section}

\begin{document}

\title[Quaternionic inverse Fourier Transforms on LCA Groups ]
{Quaternionic inverse Fourier Transforms on Locally Compact Abelian Groups}
\author[ M.J. Saadan]{Majid Jabbar Saadan}
\address{Majid Jabbar Saadan, Department of Pure Mathematics, Ferdowsi University of Mashhad, Mashhad, P.O. Box 1159-91775, Iran}
\email{majidalothman4@gamil.com}
\author[M. Janfada]{Mohammad Janfada}
\address{Mohammad Janfada, Department of Pure Mathematics, Ferdowsi University of Mashhad, Mashhad,  P.O. Box 1159-91775, Iran}
\email{janfada@um.ac.ir}
\author[ R.A. Kamyabi]{Radjab Ali Kamyabi-Gol}
\address{Radjab Ali Kamyabi-Gol, Department of Pure Mathematics, Ferdowsi University of Mashhad, Mashhad, P.O. Box 1159-91775, Iran}
\email{kamyabi@um.ac.ir}



\begin{abstract}
In this paper, the inverse of the quaternionic Fourier transform (QFT) on locally compact abelian groups is investigated. Due to the non-commutatively of multiplication of quaternions, there are different types of QFTs right, left and two-sided quaternionic Fourier transform. We focus on the right-sided quaternionic Fourier transform (RQFT) and two-sided quaternionic Fourier transform (SQFT). We establish the quaternionic Plancherel and inversion theorems for the square integrable quaternionic-valued signals on $G^2$, the space $L^2\left(G^2, \mathbb{H}\right)$, where $G$ is a locally compact abelian group. Also  RQFT on the space $L^2\left(G^2, \mathbb{H}\right)$ is studied. Furthermore relations between RQFT and SQFT are discussed. These results provide new proofs for the classical inverse Fourier transform, Plancherel theorem and etc. in $L^2(G)$.

\end{abstract}

\maketitle

\section{Introduction}\label{se1}
 In classical Fourier theory, for any $f\in L^1\left(\mathbb{R},\mathbb{C}\right)$, the Fourier transform $\hat{f}(\xi )$ is well-defined by
\begin{equation} \label{eq1.1}
\hat{f}\left(\xi \right)=\int_{\mathbb{R}} f(x)e^{-i\xi x}dx,  \quad   (\xi\in  \mathbb{R}).
\end{equation}
There are several different definitions of the classical Fourier transform known in literature, see \cite{Butz}.
Via the  inversion formula, the transform can be reversed, so that, well-behaved functions $f\ $can be represented as an infinite sum of trigonometric polynomials, where the limit of which equals                                        $f\left(x\right)=\int_{\mathbb{R}} \hat{f}\left(\xi \right)e^{i\xi x}d\xi $, for almost every $x\in  \mathbb{R}$.\\
This transform is a very powerful tool in fields such as chemistry, physics, and computer engineering. For example, complicated sound waves take the form of periodic functions, and the infinite sums that represent them can be approximated very well by just a couple of leading Fourier coefficients. Plancharel's theorem is an application of the Fourier transform that is used to analyze particles in quantum physics. This Fourier mapping and its characteristics do not stem from properties of the real numbers, but instead from certain mathematical spaces. The Fourier transform can thus be generalized to sets other than the real line, such as the circle, the integers, and in fact any locally compact abelian group (LCA group). Studying the Fourier transform of  LCA groups allows us to explain many of the properties that we take for granted about the everyday Fourier transform of real numbers.\\
 The previous contributions on inversion theorem and energy-preserved property of QFTs are developed in \cite{Cheng,Ell2,Hart}. On one hand, however, the existing results are not well established systematically. On the other hand, the prerequisites for setting up of the established theorem have not be studied completely.

 We are going to generalize the idea of the classical Fourier analysis into the quaternionic case. Due to the non-commutativity this extension is not trivial and lacks some important features. There has been a lot of interest recently in the quaternionic Fourier transform. We will compare the two best-known quaternionic Fourier tansforms, the one-sided version and the two-sided version. The one-sided version and the derived structure behind it lacks some needed properties, which the two-sided (or sandwiched) version has. But there are some triﬂes in the transfer, e.g. in the normal case we have $T_{\omega}\mathcal{F}f=\mathcal{F}{M_{\omega}f}$ whereas in quaternionic case we have $T_{\omega}\mathcal{F}_sf=\mathcal{F}_s{\overline{M_{-\omega}f}}$.\\

 In light of this, the inversion theorem on $L^1\left(G^2,\mathbb{H}\right)$ and Plancherel theorem of QFTs are investigated thoroughly in this paper, where $G$ is a locally compact abelian group. Therefore, it is of great interest to progress the function theory of QFT for for square integrable functions with respect to the locally compact abelian groups.To achieve this goal, we want to adopt the method of approximation to the identity by (good kernels). This method is commonly used in classical Fourier analysis\cite{Rud,Stein,Foll}.

For the case $G=\mathbb{R}$ the quaternion Fourier transforms (QFTs) play a vital role in the representation of (hypercomplex) signals \cite{Ell2} and in \cite{Bulow,Asef,EllS} authors used the QFT to proceed the color image analysis. E. Bayro-Corrochano, N. Trujillo, M. Naranjo in \cite{Bayro} applied the QFT to image pre-processing and neural computing techniques for speech recognition.
 Since we are working in section \ref{se3} of this paper with quaternion-valued function $f \in L^2\left(G^2,\mathbb{H}\right)$, we extend the Fourier transform. There is, as T. A. Ell already explained beautifully in case $G=\mathbb{R}$ in \cite{Ell3}, a wide range of potential definitions. The closest generalization would be to use two characters on the same side. This idea leads to the following approach, already introduced by M. Bahri and R. Ashino in case $G=\mathbb{R}$ in their Article Two-Dimensional Quaternionic Windowed Fourier transform \cite{Bahri}.
 Following the pioneering works of Ell and Sangwine, Hitzer studied the QFTs (including right-sided QFT and two-sided QFT) applied to quaternion-valued functions \cite{Hitzer1} and presented a series of further generalizations for QFTs \cite{Hitzer2,Hitzer3,Hitzer4}.
In 2016 Daniel Alpay, Fabrizio Colombo, David P. Kimsey and Irene Sabadini in \cite{Alpay}
are studied the quaternion-valued positive definite functions on locally compact abelian groups, real countably Hilbertian Nuclear spaces and on the space $\mathbb{R}^n$, endowed with the Tychonoff topology. In particular, they are proved a quaternionic version of the Bochner–Minlos theorem.\\

 This paper is organized as follows: Section \ref{se2} recalls some basic knowledge of quaternion algebra. Section \ref{se3} investigates the inversion theorem on $L^2\left(G^2,\mathbb{H}\right)$ and Plancherel theorem of right-sided QFT. In section \ref{se4}, we establish the relation between RQFT and two-sided QFT (SQFT) and study the elementary properties of SQFT.
  \section{Preliminaries}\label{se2}
In this section, we provide some basic concepts of quaternions which are essential for our study.

 In 1843, W. Hamilton discovered the new multiplication rules of a generalization of the complex numbers. He was so satisfied that he engraved them in a stone of the Broombridge road. The name of those special 4-tupel is Quaternions. This introduction to Quaternions is based on V. Kravchenkos Applied Quaternionic Analysis \cite{Kravchenko} and the book of K. G\"{u}rlebeck, K. Habetha, and W. Spr\"{o}{\ss}ig \cite{Gur}. An excellent introduction to the history and developments of quaternions were given by \cite{Adler,Ghil}.\\

  Throughout the paper, let
\[\mathbb{H}=\{q:q=a+bi+cj+dk \mbox{ with } a,b,c,d\in \mathbb{R}\},\]
be the Hamiltonian skew field of quaternions, where the elements $i,\ j$ and $k$ are imaginary units with Hamilton's multiplication rules:
\[ij=-ji=k,\  \ ki=-ik=j,\ \ jk=-kj,\ \  ii=jj=kk=-1.\]
For every quaternion $q=a+bi+cj+dk$, the scalar and vector parts of  $q$, are defined as $Sc(q)=a$ and   $vec(q)=bi+cj+dk$, respectively. If $Sc(q)=0 $, then $q$ is called pure imaginary quaternion. The set of all pure imaginary quaternions is denoted by $Im( \mathbb{H}).$ The quaternion conjugate is defined by $\bar{q}:=a-bi-cj-dk$, and the norm $|q|$ of $q$ defined as $|q|:=\sqrt{q\bar{q}} =\sqrt{a^2+b^2 +c^2+d^2}$.
Then we have $\bar{\bar{q}}\ =\ q$, $\overline{p+q}=\overline{p}+\bar{q}, \overline{pq}=\bar{q}\ \overline{p}, |pq|=|p||q|$,  for all $p,q\in \mathbb{H}$. Using the conjugate and norm of $q$, one can define the inverse of $q\ \in  \mathbb{H}  \backslash   \{0 \}$ by $q^{-1}=\frac{\bar{q}}{{\left|q\right|}^2}$.
The multiplication of two quaternions is noncommutative, but
\begin{equation} \label{eq2.1}
Sc\left(pq\right)=\ Sc\left(qp\right)\ \ (p,q\ \in\mathbb{H} ) .
\end{equation}
Put $ \mathbb{S}:=\left\{q\ \in Im\left( \mathbb{H}\right):\left|q\right|=1\right\}$, which is called the sphere of unit pure imaginary quaternion. For any $\mu \in  \mathbb{S}$, the quaternion has subsets ${\mathbb{C}}_{\mu}:=\{ \alpha+\mu \beta \in \mathbb{H}:\alpha,\beta \in \mathbb{R}\}$.  For each fixed $\mu \in \mathbb{S}$, the set ${\mathbb{C}}_{\mu }$ is isomorphic to the complex plane. Equivalently, $\mathbb{H}=\bigcup_{\mu \in \mathbb{S}}\mathbb C_{\mu }$. We denote by the set ${\mathbb{T}}_{{\mathbb{C}}_\mu}={\left\{q\ \in {{\mathbb{C}}_\mu}:\left|q\right|=1\right\}}$.

The space $L^p\left( \mathbb{R}^2,\mathbb{H}\right)$, $1\le p<\infty$ is considered in \cite{Hart}. By a similar argument, for a locally compact abelian group $G,$  we may define the space $L^p\left(G^2,\mathbb{H}\right)$ as follows.\\
The space $L^p\left(G^2,\mathbb{H}\right), 1\le p<\infty$, is the left module of all quaternion-valued measurable functions  $\ f:G^2\to \mathbb{H}$ with the finite norm:
\[{\left\|f\right\|}_p\ = (\int_{G^2} \left|f\left(x_1,x_2\right)\right|^p {d^2}_{ \mu _{G^2}}\left(x_1,x_2\right)) ^{\frac{1}{p}}\ <\infty,\]
where $ d^2_{ \mu_{G^2}}\left(x_1,x_2\right)=d_{ \mu  _{G}}x_1d_{ \mu_{G}}x_2$ represents the Haar measure on $G^2$.\\
For $p=\infty$, the space $L^{\infty }\left(G^2,\mathbb{H}\right)$ is defined by
\[L^{\infty }\left(G^2,\mathbb{H}\right)=\left\{f:G^2\longrightarrow \mathbb{H}:f  \mbox{\ is \ measurable \ and } \left\|f\right\| _{\infty } <\infty \right\},\]
where
\[ \left\|f\right\|_{\infty}=ess \sup_{\left(x_1,x_2\right)\in G^2} \left(\left|f\left(x_1,x_2\right)\right|\right).\]

With a similar argument of \cite{Hart} we may define a real inner product on $L^2(G^2, \mathbb{H} )$ as follows:
\begin{align*}
\langle f,g\rangle &=  \frac{1}{2}\int_{G^2}((g(x_1,x_2)\overline{f(x_1,x_2)}+f(x_1,x_2)\overline{g(x_1,x_2)})d^2_{{\mu }_{G^2}}\left(x_1,x_2\right)\\
&=Sc\int_{G^2} f\left(x_1,x_2\right)\overline{g\left(x_1,x_2\right)}{d^2}_{{\mu }_{G^2}}\left(x_1,x_2\right).
\end{align*}
  It is also possible to define an inner product on $L^2(G^2, \mathbb{H})$ by
  \begin{align}\label{RRR}
\left(f,g\right)= \int_{G^2} f\left(x_1,x_2\right)\overline{g\left(x_1,x_2\right)}{d^2}_{{\mu }_{G^2}}\left(x_1,x_2\right).
\end{align}
Clearly $ \langle f,g \rangle =Sc\left(f,g\right),$ and the induced norms of  $\langle \cdot, \cdot\rangle $, and $( \cdot,\cdot)$ are equals.\\
It is not difficult to verify that
\begin{equation}\label{eq2.2}
\left(pf,qg\right)=p \left(f,g\right)\bar{q} \ \ \ \     (f,g\in L^2(G^2,\mathbb{H} ),\ \ p,q\in \mathbb{H}).
\end{equation}

\begin{remark}\label{re2.1}
For any $f,g\in L^p\left(G^2,\mathbb{H}\right),1\le p\le \infty$, we have
\begin{enumerate}
\item[(i)]
$ f=f_{0}+f_1 i+f_2 j+f_3 k $, where $ f_m \in L ^p\left(G^2, \mathbb{R}\right),  m=0,1,2,3,$ and $ f\in L ^p\left(G^2, \mathbb{H}\right)$ if and only if $ f_m\in L^p \left(G^2, \mathbb{R}\right)$, for $m=0,1,2,3.$
\item[(ii)]
From the fact that any $q\in\mathbb{H}$ has the form $q=x+yj$, for some $x,y\in \mathbb{C}$, we may conclude $ f=f_1+f_2j$, where ${f_1,f_2\in L}^p\left(G^2, \mathbb{C}\right)$.
\item[(iii)]
Similar to the part (ii), for any $  \mu \in \mathbb{S}$, $f=f_1+f_2\mu$, for some  $f_1,f_2\in L^p\left(G^2, \mathbb{C}\right)$. Thus, we can consider every $f\in L^p\left(G^2, \mathbb{H}\right),\ (1\le p\ \le \infty )$ as a linear combination of real (complex) $L^p$-functions.
\item[(iv)]
If  $f=f_0+f_1i+f_2j+f_3k$ and $g=g_0+g_1i+g_2 j+g_3 k$, then one can see that
\begin{align*}
\langle f,g\rangle =& Sc(\int_{G^2}{\ f_0\left(x_1,x_2\right)g_0\left(x_1,x_2\right)}\ {d^2}_{{\mu }_{G^2}}\left(x_1,x_2\right)\\
&+\int_{G^2}{f_1\left(x_1,x_2\right)g_1\left(x_1,x_2\right)}\ {d^2}_{{\mu }_{G^2}}\left(x_1,x_2\right)\\
&+\int_{G^2}{\ f_2\left(x_1,x_2\right)g_2\left(x_1,x_2\right)}\ {d^2}_{{\mu }_{G^2}}\left(x_1,x_2\right)\\
&+\int_{G^2}{\ f_3\left(x_1,x_2\right)\ g_3\left(x_1,x_2\right)}\ {d^2}_{{\mu }_{G^2}}\left(x_1,x_2\right)).
\end{align*}
\end{enumerate}
For $f=f_0+f_1i+f_2j+f_3k$, we have ${\left|f\right|}^2=\sum^3_{m=0}{{\left|f_m\right|}^2}$ and therefore
\[{\left|f\right|}\le 2{\max  \left\{{\left|f_m\right|}:m=0,1,2,3\right\}}\le 2\sum^3_{m=0}{\left|f_m\right|}.
\]
Thus ${\left\|f\right\|}_{\infty }\le 2\sum^3_{m=0}{{\left\|f_m\right\|}_{\infty }},\ {\left\|f\right\|}_1\le 2\sum^3_{m=0}{{\left\|f_m\right\|}_1},$ and ${\left\|f\right\|}^2_2\le \sum^3_{m=0}{{\left\|f_m\right\|}^2_2}$.\\
Therefore, some properties of real (complex) $L^p-$functions can be naturally  extended to quaternionic $L^p$-functions.
\end{remark}
We refer to the usual text books about locally compact groups\cite{Feichtinger,Fischer,Foll,Rud,Tur,Ruz,Stein,Deitmar}.
In the following, we introduce the concept of a quaternionic character on $G^2$.
\begin{definition}\label{de2.2}
  Let $G$ be a locally compact abelian group. For any continuous characters  $\omega_i:G \to   \mathbb{T}_{ \mathbb{C}_{i}}$ and $\omega _j:G \to {\mathbb{T}}_{{\mathbb{C}}_j}$ define $\omega :G^2\to {\mathbb{T}}_Q$ by
  \begin{equation} \label{eq2.3}
\omega \left(x_1,x_2\right)=\omega_i\left(x_1\right)\omega_j\left(x_2\right),\ \ \ (x_1,x_1\in G)
\end{equation}
where ${\mathbb{T}}_Q~:=\left\{q\ \in  \mathbb{H}:\left|q\right|=1\right\}$. We call $\omega $  a $\mathbb{H}$-valued character of $G^2$ and the set of all $\mathbb{H}$-valued characters of the form $\omega $ is denoted by $\widehat{G^2}$. The set $\widehat{G^{2}}$ is called the quaternionic dual group of $G^2$.
\end{definition}
Letting  $\widehat{G_{\mathbb{C}_i}}$ and  $\widehat{G_{\mathbb{C}_j}}$  the set of all characters of the form $\omega_i:G\to {\mathbb{T}}_{{\mathbb{C}}_i}$ and ${\omega }_j:G\to {\mathbb{T}}_{{\mathbb{C}}_j}$, respectively,  we get
$\widehat{G^2}=\widehat{G_{\mathbb{C}_i}}$ $\times $ $\widehat{G_{\mathbb{C}_j}}$. But form the fact that $\mathbb{C}_i\cong \mathbb{\mathbb{C}}$ and $\mathbb{C}_j\cong \mathbb{C}$, we get $\widehat{G_{\mathbb{C}_i}}$ $\cong \widehat{G}$  and $\widehat{G_{\mathbb{C}_j}}\cong \widehat{G}$, where $\widehat{G}$ is the so called dual of $G$. Hence  we may consider $\widehat{G^2}$ as a topological group with its natural structure.

\begin{definition}[\cite{Kam}]\label{de2.3}
  Let $G$ be a second countable LCA group. For a topological automorphism $\alpha $ on $G^2$, we say ${\alpha }^{-1}$ is contractive if,  for every compact subset $K$ of $G^2$ and any neighborhood $U$ of the identity, there exists a positive integer $N$, depending on $K$ and $U$, such that $\alpha^{-l}(K)\ \subseteq \ U $ for any $ l>N.$

Let ${\alpha }^{-1}_i$ and ${\alpha }^{-1}_j\ $  be contractive with respect to the automorphisms ${\alpha }_i$ on ${\hat{G}}_{{\mathbb{C}}_i}$, and the automorphism ${\alpha }_j\ $ on ${\hat{G}}_{{\mathbb{C}}_j}$, respectively, and let ${\Phi }_1\in L^1({\hat{G}}_{{\mathbb{C}}_i}, \mathbb{R}{\rm )}\cap C_0({\hat{G}}_{{\mathbb{C}}_i}, \mathbb{R}{\rm )}$, ${\Phi }_2\in L^1({\hat{G}}_{{\mathbb{C}}_j}, \mathbb{R}{\rm )}\cap C_0({\hat{G}}_{{\mathbb{C}}_j}, \mathbb{R}{\rm )})$.  We say $\Phi:=(\Phi_1, \Phi_2)$ is increasing to 1 with respect to the $(\alpha_i, \alpha_j)$, if
\[{\mathop{\lim }_{l\longrightarrow \infty } {\Phi }_1({\alpha }^{-l}_i\left(\omega_i\right))={\Phi }_1\left(0\right)=1\ }\mbox{ and }{\mathop{\lim }_{l\longrightarrow \infty } {\Phi }_2({\alpha }^{-l}_j\left(\omega_j\right))={\Phi }_2\left(0\right)=1, }\]
for every $\omega_i\in \widehat{G_{\mathbb{C}_i}}$, and $\omega_j\in \widehat{G_{\mathbb{C}_j}}$.\\
Given $l\in \mathbb{N}$ and $\left(x_1,x_2\right)\in G^2$, set
\begin{align*}
&P^l_i\left(x_1\right):=\int_{{\hat{G}}_{{\mathbb{C}}_i}}{{\Phi }_1({\alpha }^{-l}_i\left(\omega_i\right))\omega_i\left(x_1\right)d\omega_i}\,\\
&P^l_j\left(x_2\right):=\int_{{\hat{G}}_{{\mathbb{C}}_j}}{{\Phi }_2({\alpha }^{-l}_j\left(\omega_j\right))\omega_j\left(x_2\right)d\omega_j},
\end{align*}
  and put $P^l\left(x_1,x_2\right)=P^l_i\left(x_1\right)P^l_j\left(x_2\right)$. Then
\[\int_{G^2}{P^l_i\left(x_1\right)P^l_j\left(x_2\right){d^2}_{{\mu }_{G^2}}\left(x_1,x_2\right)}={\Phi }_1\left(0\right){\Phi }_2\left(0\right)=1.\]
\end{definition}

For example, when $G=\mathbb{R}$ (see \cite{Cheng}), one may consider $\alpha$  the automorphism on $\mathbb{R}$  defined
$ \alpha(\omega)=2\omega$ which implies that $\alpha^{-l}(\omega)=2^{-l}\omega$. Also we may consider
 \[
 \Phi_1 (\omega)=\Phi_2(\omega)=e^{-|\omega|}.
 \]
 Therefore
 \[
 \Phi_k\big(\alpha^{-l}(\omega)\big)=e^{-|2^{-l}\omega|}, \quad \Phi_k(0)=1, \quad k=1,2.
 \]
 Putting $\epsilon_l=2^{-l}_1$, we see that $\epsilon_l\to 0$ as $l\to \infty$. Also one may show  from definition of $P^l_i$ and $P^l_j$ that
 \[ P^l_i\left(x_1\right)=\frac{1}{\pi }\frac{\epsilon_l}{(\epsilon_l^2+x_1)}, \mbox{ and } P^l_j\left(x_1\right)=\frac{1}{\pi }\frac{\epsilon_l}{(\epsilon_l^2+x_2)},\ l>0.\]
 Thus
  \[ P^l\left(x_1,x_2\right)=P^l_i\left(x_1\right)P^l_j\left(x_2\right)=\frac{1}{{\pi }^2}\frac{\epsilon_l^2}{(\epsilon_l^2+x_1)(\epsilon_l^2+x_2)},\ l>0,\] which is so-called the Poisson kernel.
We need the following result of \cite{Kam} for the group $G^2$.
\begin{lemma}[\cite{Kam}]\label{le2.4}
Let $\alpha \in Aut(G)$ be contractive and let $K$ be a closed neighborhood of $e$. For every $l\in \mathbb{N}$, let $K_l:=\ \bigcap\{{\alpha }^{-k}\left(K\right),\ k\ge l,\ k\in \mathbb{N}\}$. Then
\begin{enumerate}
\item[(i)] $K_l\supset K_{l+1}$ and ${\alpha }^{-l}\left(K_l\right)=K_{l+1}$ for all $l\in \mathbb{N}$;

\item [(ii)] $\bigcup_{l\in \mathbb{N}}{{\alpha }^{-l}\left(K_l\right)}=G$.
\end{enumerate}
\end{lemma}
Let us begin with the following lemma. As usual, we denote by $L_{\left(y_1,y_2\right)}g$ and $R_{\left(y_1,y_2\right)}g$ the left and right translation of $g$ on $G^2$ respectively.
\begin{lemma}\label{le2.5}
For given $1\ \le \ p<\infty$, and $g\in L^p(G^2, \mathbb{H})$, the map $(y_1,y_2)\ \mapsto \ L_{(y_1,y_2)}g$ is continuous from $G^2\ $ to $L^p(G^2, \mathbb{H}{\rm )}$. In other words, for every $\varepsilon >0,$ there exists a neighborhood $V$ of the zero such that ${\left\|L_{\left(y_1,y_2\right)}g-g\right\|}_p$ and ${\left\|R_{(y_1,y_2)}g-g\right\|}_p$ tends to zero as $(y_1,y_2)\longrightarrow(0,0)$ and for any $(y_1,y_2)\in V$.
\end{lemma}
\begin{proof}
Firstly, Given $g\in C_c\left(G^2, \mathbb{H}\right)$ and $\epsilon>0$, let $ K = suppg$. For every $(x_1,x_2)\in K$ there is a neighbourhood $U_{(x_1,x_2)}$ of $(0,0)$ such that
${\left\|g(x_1+y_1,x_2+y_2)-g(x_1,x_2)\right\|}<\frac{1}{2}\epsilon$ for $(y_1,y_2)\in U_{(x_1,x_2)}$, and there is a symmetric neighbourhood $V_{(x_1,x_2)}$ of $(0,0)$ such that  $V_{(x_1,x_2)}V_{(x_1,x_2)}=U_{(x_1,x_2)}$. The sets $(x_1,x_2)V_{(x_1,x_2)} \quad (x_1,x_2)\in K$ cover $K$, so there exist $(x_1,x_2)_1,...,(x_1,x_2)_n \in K$ such that $K\subset \bigcup_{a=1}^{n}(x_1,x_2)_a V_{(x_1,x_2)_a}$.\\
Let $V=\bigcap_{a=1}^{n}V_{(x_1,x_2)_a}$; we claim that ${\left\|R_{\left(y_1,y_2\right)}g-g\right\|}_\infty<\epsilon$ for any $(y_1,y_2)\in V$.\\
If $(x_1,x_2)\in K$, then there is some $a$ for which ${(-x_1,-x_2)_a+(x_1,x_2)\in V_{(x_1,x_2)_a}}$, so that $(x_1+y_1,x_2+y_2)=(x_1,x_2)_a + (-x_1,-x_2)_a+(x_1,x_2)+(y_1,y_2)\in (x_1,x_2)_aU_{(x_1,x_2)_a}$.  But then \\
\begin{align*}
&{\left\|g(x_1+y_1,x_2+y_2)-g(x_1,x_2)\right\|}\\
&\leq{\left\|g(x_1+y_1,x_2+y_2)-g((x_1,x_2)_{a})\right\|}\\
&+{\left\|g((x_1,x_2)_{a})-g(x_1,x_2)\right\|}<\frac{1}{2}\epsilon+\frac{1}{2}\epsilon=\epsilon.
\end{align*}
Similarly, if $(x_1+y_1,x_2+y_2) \in K$ then ${\left\|g(x_1+y_1,x_2+y_2)-g(x_1,x_2)\right\|}<\epsilon$. But if $(x_1+y_1,x_2+y_2)$ and $(x_1,x_2)$ are not in $K$ then $g(x_1+y_1,x_2+y_2)=g(x_1,x_2)=0$, so we are done.\\
Now, for proof lemma, Fix a compact neighbourhood $V$ of $(0,0$). First, we can choose $f\in C_c\left(G^2, \mathbb{H}\right)$, let $K = (suppg)V  \bigcup V (supp g)$. Then $K$ is compact, and $R_{\left(y_1,y_2\right)}g$ and $L_{\left(y_1,y_2\right)}g$ are supported in $K$ when $(y_1,y_2)\in V$ .  Hence, \\
${\left\|L_{\left(y_1,y_2\right)}f-f\right\|}_p\leqslant {\left|K\right|}^ \frac{1}{p}{\left\|L_{\left(y_1,y_2\right)}f-f\right\|}_\infty\longrightarrow0$ as $(y_1,y_2) \longrightarrow(0,0)$, and likewise ${\left\|R_{(y_1,y_2)}g-g\right\|}_p$ tends to zero.\\
Now suppose $g\in L^p\left(G^2, \mathbb{H}\right)$. We have ${\left\|L_{\left(y_1,y_2\right)}g\right\|}_p={\left\|g\right\|}_p$ and ${\left\|R_{\left(y_1,y_2\right)}g\right\|}_p=\Delta(y_1,y_2)^{\frac{-1}{p}}{\left\|g\right\|}_p\leq C{\left\|g\right\|}_p$  for $(y_1,y_2)\in V$, where $\Delta:G^2\longrightarrow\mathbb{R}^{+}$ be a modular function.\\ Given $\epsilon>0$ we can choose $f\in C_c\left(G^2, \mathbb{H}\right)$ such that
${\left\|g-f\right\|}_p\leq\epsilon$, and then\\
${\left\|R_{\left(y_1,y_2\right)}g{\rm \ }{\rm -}g\right\|}_p ={\left\|R_{\left(y_1,y_2\right)}(g-f)\right\|}_p +{\left\|R_{\left(y_1,y_2\right)}f{\rm \ }{\rm -}f\right\|}_p+{\left\|(f-g)\right\|}_p \leq(C+1)\epsilon+{\left\|R_{\left(y_1,y_2\right)}g{\rm \ }{\rm -}g\right\|}_p$, and the last term tends to zero as $(y_1,y_2) \longrightarrow(0,0)$, and likewise technique for ${L_{(y_1,y_2)}f}$.
\end{proof}
The next theorem underlies many of the important applications of convolutions of $\mathbb{H}$-valued functions on $G^2$.
\begin{theorem}\label{th2.6}
Suppose that $\Phi=(\Phi_1,\Phi_2):\widehat{G^2}\to \mathbb{R}$, is increasing to 1 with respect to $(\alpha_i, \alpha_j)$, then
\begin{enumerate}
\item[(i)]
${\mathop{\lim }_{l\to \infty } {\left\|f*P^l -f\right\|}_p=0\ }$ for every $f\in L^p(G^2, \mathbb{H})$, $p=1,2$;
\item[ (ii)]
${\mathop{\lim }_{l\to \infty } f*P^l\left(x_1,x_2\right)=f\left(x_1,x_2\right)}$ if $f\left(x_1,x_2\right)\in L^{\infty }(G^2, \mathbb{H})$ is continuous at a point $\left(x_1,x_2\right)$,
\end{enumerate}
Where $P^l$ and $\Phi$ are defined in \ref{de2.3}.
\end{theorem}
\begin{proof}
Since, from definition of $P^l$ and $\Phi$
\begin{align*}
&\int_{G^2}{P^l(x_1,x_2){d^2}_{{\mu }_{G^2}}\left(x_1,x_2\right)}\\
&=\int_{G^2}{P^l_i\left(x_1\right)P^l_j\left(x_2\right){d^2}_{{\mu }_{G^2}}\left(x_1,x_2\right)}\\
&={\Phi }_1\left(0\right){\Phi }_2\left(0\right)=1,
\end{align*}
we have
\begin{align*}
&f*P^l(x_1,x_2) -f(x_1,x_2)\\
&=\int_{G^2}{P^l(-y_1,-y_2)f(x_1+y_1,x_2+y_2){d^2}_{{\mu }_{G^2}}\left(y_1,y_2\right)}\\
&-f(x_1,x_2)\int_{G^2}{P^l(y_1,y_2){d^2}_{{\mu }_{G^2}}\left(y_1,y_2\right)}\\
&=\int_{G^2}{(L_(y_1,y_2)f(x_1,x_2)-f(x_1,x_2))P^l(y_1,y_2){d^2}_{{\mu }_{G^2}}\left(y_1,y_2\right)}.
\end{align*}
So by Minkowski's inequality for integrals, \\
\begin{align*}
&{\left\|f*P^l-f\right\|}_p\\
&\leq \int_{G^2}{{\left\|R_{(y_1,y_2)}f-f\right\|}_p P^l(y_1,y_2){d^2}_{{\mu }_{G^2}}\left(y_1,y_2\right)}\\
&\leq Sup_{(y_1,y_2)\in V}{{\left\|R_{(y_1,y_2)}f-f\right\|}_p}.
\end{align*}
Hence by using Lemma \ref{le2.5} we get ${\left\|f*P^l-f\right\|}_p$ tend to zero as $U$ tend to zero, and from definition of $p^l$, we get the assertion(i).
The second assertion follows in the same way in Lemma 1.6.5 in \cite{Deitmar}.
\end{proof}

\begin{lemma}\label{le2.7}
If $\{f_n\}$ is a Cauchy sequence in $L^p(G^2, \mathbb{H}{\rm )}$, $p =1,2,\infty$, with limit $f$, then $\{f_n\}$ has a subsequence which converges pointwise for almost every $\left(x_1,x_2\right)\in \ G^2\ $to $f.$
\end{lemma}
\begin{proof}
Let ${\lbrace f_n\rbrace}$ be a Cauchy sequence in $L^p(G^2, \mathbb{H}{\rm )}$. For each $n$, ${\lbrace f_n\rbrace}$ can be written as $f_{n}=f_{n}^1+if_{n}^2+jf_{n}^3+kf_{n}^4$ where the $f_{n}^e \in L^p(G^2, \mathbb{R}{\rm )}, e=1,2,3,4$. From Minkowski' s inequality, we get\\
${\left\|f+g\right\|}_p\leq {\left\|f\right\|}_p+{\left\|g\right\|}_p$, it follows that each of the sequences $f_{n}^e ,\quad e=1,2,3,4$ are Cauchy sequences in $L^p(G^2, \mathbb{R}{\rm )}$.\\
By similar technique, we can see that $(L^p(G^2, \mathbb{R}), {\left\|.\right\|}_p)$ is complete for general case(see theorem 3.11 in \cite{RudR}) then there exists function $f^e ,\quad e=1,2,3,4$ in $L^p(G^2, \mathbb{R})$ such that ${\mathop{\lim }_{n} {f_{n}^e=f^e, e=1,2,3,4}}$. Applying the Minkowski inequality again it follows that the function $f=f^1+if^2+jf^3+kf^4$ is in $L^p(G^2, \mathbb{H}{\rm )}$ and ${\mathop{\lim }_{n} {f_{n}=f}}$. Therefore, $L^p(G^2, \mathbb{H})$ is complete. Then quaternionic Riesz-Fischer theorem is hold.
Thus, $(L^p(G^2, \mathbb{H}),{\left\|.\right\|}_p)$, $1 \leq p \leq \infty$ is Banach space.\\
Now, let ${\lbrace f_{n}\rbrace}_{n \in \mathbb{N}}$ be a Cauchy sequences in $L^p(G^2, \mathbb{H})$, Choose indices $n_1, n_2,.....$ so that ${\left\|f_{n_{k+1}}-f_{n_{k}}\right\|}_p < \frac{1}{2^k}, \ k\in \mathbb{N}$ and put $g_{k}=f_{n_{k+1}}-f_{n_{k}}$. Let $g=\sum_{k=1}^{\infty}{\left|g_{k}\right|}$. By Minkowski inequality we have $\sum_{k=1}^{m}{\left\|g_k -f\right\|}_p$ is integrable with integrands bound above by ${(\sum_{k=1}^{m}{\left\|g_k -f\right\|}_{p})^p} \leq 1 $. It then follows from the Monotone Convergence Theorem that $g^p$ is integrable, hence $g \in L^p(G^2, \mathbb{H})$.   Define the following set $N = {\lbrace (x_1,x_2)\in G^2 : f(x_1,x_2)=\infty \rbrace} \subset G^2$, then from remark B 3.1 in \cite{Deitmar}, we get $g(x_1,x_2)\neq \infty$ outside a null set$N$, and the series $\sum_{k=1}^{m}{g_k}$ converges absolutely to some function $h(x_1,x_2)$ for every $(x_1,x_2)\notin N $. We trivially extend $h$ to all of $G^2$. Then $h$ is measurable as a pointwise limit of measurable functions, and since ${\left|h(x_1,x_2)\right|}\leqslant g(x_1,x_2)$ for every
$(x_1,x_2) \in G^2$ we have $h \in L^p(G^2, \mathbb{H})$. Put $f=f_{n_{1}}+h$. Then
\begin{align*}
&f(x_1,x_2)=f_{n_{1}}(x_1,x_2)+\lim_{m\longrightarrow\infty}{\sum_{k=1}^{m}{(f_{n_{k+1}}-f_{n_{k}})}}\\
&=\lim_{m\longrightarrow\infty}{\sum_{k=1}^{m}{f_{n_{m}}(x_1,x_2)}} for every (x_1,x_2)\notin N.
\end{align*}
Which is complete prove.
\end{proof}
It is not difficult to see that the space $L^2\left(G^2, \mathbb{H}\right)$ with the inner product in (\ref{RRR}) is the left $\mathbb{H}$-module Hilbert space.
\section{Fourier analysis on $\mathbf{G}^{2}$ over $ \mathbb{H}$}\label{se3}
\subsection{Right-Sided QFT on locally compact abelian groups}\label{su3.1}
The quaternion Fourier transform (QFT) is first defined by Ell to analyze linear time invariant systems of partial differential equations \cite{Ell1}. An excellent introduction to the history and developments of QFT was given by Brackx et al \cite{Brackx}. The noncommutatively of the quaternion multiplication leads to different types of quaternion Fourier transformations (see \cite{Hart}). In this section, we consider the right-sided quaternion Fourier transformation (RQFT) on locally compact abelian group.

The RQFT of $f\in L^2({\mathbb{R}}^2, \mathbb{H}{\rm )}$ is considered in \cite{Hart}. By a similar argument, we may define the RQFT of $f\in L^2\left(G^2, \mathbb{H}\right)$, which is a function from $\widehat{G^2}\ $ to $\mathbb{H}$ as follows:
\[{\mathcal{F}}_r\left(f\right)\left(\omega_i,\omega_j\right)=\ \hat{f}\left(\omega_i,\omega_j\right)=\int_{G^2}{f\left(x_1,x_2\right)\overline{\omega_i(x_1)}\ \overline{\omega_j{(x}_2)}}\ \ {d^2}_{{\mu }_{G^2}}\left(x_1,x_2\right).\]
We are now ready to invert the RQFT. If $f\in L^{1}\left(G^2, \mathbb{H}\right)$, then we define
\[{\mathcal{F}_r}^{-1}f\left(x_1,x_2\right)=\mathcal{F}_rf\left(-x_1,-x_2\right)=
\int_{\widehat{G^2}}{f\left(\omega_i,\omega_j\right)}\omega_j\left(x_2\right)\omega_i
\left(x_1\right)d\omega_id\omega_j.\]
We claim that if $f\in L^{1}\left(G^2, \mathbb{H}\right)$ and $\mathcal{F}_rf\in L^{1}\left(\widehat{G^2}, \mathbb{H}\right)$, then ${\mathcal{F}_r}^{-1}(\mathcal{F}_rf)=f$.

The following results are related to our study of the inversion theorem and Plancherel's theorem of RQFT. We use the integral representations to express the convolutions.
\begin{proposition}\label{pr3.2}
Let $f=f_1+f_2j\in L^1\left(G^2, \mathbb{H}\right)$. Put $\tilde{f}\left(x_1,x_2\right)=\overline{f\left(-x_1,-x_2\right)}$ and define $g\left(x_1,x_2\right)=\left(\tilde{f}*f\right)\left(x_1,x_2\right);$ then
\begin{align*}
 g\left(x_1,x_2\right)&=\int_{G^2}{\overline{f\left(y_1,y_2\right)}f\left(y_1+x_1,y_2+x_2\right)}\ {d^2}_{{\mu }_{G^2}}\left(y_1,y_2\right),\\
(f*{{\rm P}}^l)(x_1,x_2)&=[\int_{\widehat{G^2}}{[{\Phi }_1({\alpha }^{-l}_i(\omega_i)){\Phi }_2({\alpha }^{-l}_j(\omega_j)][\omega_i(x_1){\mathcal{F}}_r(f_1)(\omega_i,\omega_j)\omega_j\left(x_2\right)]}\\
&\quad +\overline{\omega_i\left(x_1\right)}{\mathcal{F}}_r\left(f_2j\right)\left(\omega_i,\omega_j\right)\omega_j\left(x_2\right))]]{d^2}_{{\mu }_{\widehat{G^2}}}\left(\omega_i,\omega_j\right)
\end{align*}
and
\[ Sc\left(\left(f*{{\rm P}}^l\right)\left(0,0\right)\right)=\int_{\widehat{G^2}}{{{\Phi }_1({\alpha }^{-l}_i\left(\omega_i\right)){\Phi }_2({\alpha }^{-l}_j\left(\omega_j\right))\left|{\mathcal{F}}_r\left(f\right)\left(\omega_i,\omega_j\right)\right|}^2{d^2}_{{\mu }_{\widehat{G^2}}}\left(\omega_i,\omega_j\right)},
\] where$\ {\Phi }_1, {\Phi }_2\ $and $P^l\ $are define in Definition \ref{de2.3}.
\end{proposition}
\begin{proof}
We compute
\begin{align*}
&P^l\left(x_1-y_1,x_2-y_2\right)\\
&=P^l_i\left(x_1-y_1\right)P^l_j\left(x_2-y_2\right)\\
&=\int_{{\hat{G}}_{{\mathbb{C}}_i}}{{\Phi }_1\left({\alpha }^{-l}_i\left(\omega_i\right)\right)}\omega_i(x_1)\overline{\omega_i(y_1)}d\omega_i\int_{{\hat{G}}_{{\mathbb{C}}_j}}{{\Phi }_2({\alpha }^{-l}_j\left(\omega_j\right))}\omega_j(x_2)\overline{\omega_j(y_2)}d\omega_j\\
&=\int_{{\hat{G}}_{{\mathbb{C}}_i}}{\int_{{\hat{G}}_{{\mathbb{C}}_j}}{{\Phi }_1({\alpha }^{-l}_i\left(\omega_i\right)){\Phi }_2({\alpha }^{-l}_j\left(\omega_j\right))\omega_i(x_1)\overline{\omega_i(y_1)}\ \overline{\omega_j(y_2)}\omega_j(x_2)d\omega_id\omega_j}}\\
&=\int_{\widehat{G^2}}{{\Phi }_1({\alpha }^{-l}_i\left(\omega_i\right)){\Phi }_2({\alpha }^{-l}_j\left(\omega_j\right))\omega_i(x_1)\overline{\omega_i(y_1)}\ \overline{\omega_j(y_2)}
\omega_j(x_2)d^2(\omega_i,\omega_j}).
\end{align*}
Now, since
\begin{align*}
&\left(f*{{\rm P}}^l\right)\left(x_1,x_2\right)\\
&=\int_{G^2}{f\left(y_1,y_2\right)P^l\left(x_1-y_1,x_2-y_2\right)dy_1dy_2}\\
&=\int_{G^2}\int_{\widehat{G^2}}f\left(y_1,y_2\right)\ {\Phi }_1({\alpha }^{-l}_i\left(\omega_i\right)){\Phi }_2({\alpha }^{-l}_j\left(\omega_j\right))\omega_i(x_1)\overline{\omega_i\left(y_1\right)}\\
&\quad \overline{\omega_j\left(y_2\right)}\omega_j(x_2){d^2}_{{\mu }_{\widehat{G^2}}}(\omega_i,\omega_j){d^2}_{{\mu }_{G^2}}\left(y_1,y_2\right)\\
&=\int_{G^2}\int_{\widehat{G^2}}{\Phi }_1({\alpha }^{-l}_i\left(\omega_i\right)){\Phi }_2({\alpha }^{-l}_j\left(\omega_j\right))(f_1\left(y_1,y_2\right)+f_2\left(y_1,y_2\right)j)\ \omega_i(x_1)\overline{\omega_i\left(y_1\right)}\\
&\quad \overline{\omega_j\left(y_2\right)}\omega_j(x_2)
{d^2}_{{\mu }_{\widehat{G^2}}}\left(\omega_i,\omega_j\right){d^2}_{{\mu }_{G^2}}\left(y_1,y_2\right)\\
&=\int_{\widehat{G^2}}\int_{G^2}\ {\Phi }_1({\alpha }^{-l}_i\left(\omega_i\right)){\Phi }_2({\alpha }^{-l}_j\left(\omega_j\right))\omega_i\left(x_1\right)f_1\left(y_1,y_2\right)\overline{\omega_i\left(y_1\right)}
\end{align*}
\begin{align*}
&\overline{\omega_j\left(y_2\right)}\omega_j\left(x_2\right){d^2}_{{\mu }_{G^2}}\left(y_1,y_2\right){d^2}_{{\mu }_{\widehat{G^2}}}\left(\omega_i,\omega_j\right)\\
&+\int_{\widehat{G^2}}\int_{G^2}\ {\Phi }_1({\alpha }^{-l}_i\left(\omega_i\right)){\Phi }_2({\alpha }^{-l}_j\left(\omega_j\right))\overline{\omega_i\left(x_1\right)}f_2
\left(y_1,y_2\right)j\overline{\omega_i\left(y_1\right)}\ \overline{\omega_j\left(y_2\right)}\\
&\omega_j\left(x_2\right) {d^2}_{{\mu }_{G^2}}\left(y_1,y_2\right){d^2}_{{\mu }_{\widehat{G^2}}}\left(\omega_i,\omega_j\right)\\
&=\int_{\widehat{G^2}}({{\Phi }_1({\alpha }^{-l}_i\left(\omega_i\right)){\Phi }_2({\alpha }^{-l}_j\left(\omega_j\right))
((\omega_i\left(x_1\right){\mathcal{F}}_r\left(f_1\right)\left(\omega_i,\omega_j\right)\omega_j\left(x_2\right)}\\
&+\overline{\omega_i\left(x_1\right)}{\mathcal{F}}_r\left(f_2j\right)\left(\omega_i,\omega_j\right)\omega_j\left(x_2\right)))){d^2}_{{\mu }_{\widehat{G^2}}}\left(\omega_i,\omega_j\right).
\end{align*}
Note that
\begin{align*}
g\left(x_1,x_2\right)&=\left(\tilde{f}*f\right)\left(x_1,x_2\right)\\
&=\int_{G^2}{\overline{f\left(-y_1,-y_2\right)}f\left(x_1-y_1,x_2-y_2\right)}\ {d^2}_{{\mu }_{G^2}}\left(y_1,y_2\right)\\
&=\int_{G^2}{\overline{f\left(y_1,y_2\right)}f\left(x_1+y_1,x_2+y_2\right)}\ {d^2}_{{\mu }_{G^2}}\left(y_1,y_2\right).
\end{align*}
Then
\begin{align*}
& Sc\left(\left(f*{{\rm P}}^l\right)\left(0,0\right)\right)=Sc(\int_{G^2}{g\left(y_1,y_2\right){{\rm P}}^l\left({-y}_1,-y_2\right){d^2}_{{\mu }_{G^2}}\left(y_1,y_2\right)})\\
&=Sc(\int_{G^2}{[\int_{G^2}{\overline{f\left(s_1,s_2\right)}}f\left(s_1+y_1,s_2+y_2\right){d^2}_{{\mu }_{G^2}}\left(s_1,s_2\right)]{{\rm P}}^l\left({-y}_1,-y_2\right){d^2}_{{\mu }_{G^2}}\left(y_1,y_2\right)})\\
&=Sc(\int_{G^2}{\int_{G^2}{\overline{f\left(s_1,s_2\right)}}f\left(s_1+y_1,s_2+y_2\right){{\rm P}}^l\left({-y}_1,-y_2\right){d^2}_{{\mu }_{G^2}}\left(s_1,s_2\right){d^2}_{{\mu }_{G^2}}\left(y_1,y_2\right)})\\
&=Sc(\int_{G^2}{\int_{G^2}{\overline{f\left(s_1,s_2\right)}}f\left(z_1,z_2\right){{\rm P}}^l\left(s_1{-z}_1,{s_2-z}_2\right){d^2}_{{\mu }_{G^2}}\left(s_1,s_2\right){d^2}_{{\mu }_{G^2}}\left(z_1,z_2\right)})\\
&=Sc(\int_{G^4}\int_{\widehat{G^2}}\overline{f\left(s_1,s_2\right)}f\left(z_1,z_2\right){d^2}_{{\mu }_{G^2}}\left(z_1,z_2\right){\Phi }_1({\alpha }^{-l}_i\left(\omega_i\right)){\Phi }_2({\alpha }^{-l}_j\left(\omega_j\right))\overline{\omega_i\left(z_1\right)}\ \overline{\omega_j\left(z_2\right)}\\
&\quad \omega_j\left(s_2\right)\omega_i\left(s_1\right){d^2}_{{\mu }_{G^2}}\left(s_1,s_2\right){d^2}_{{\mu }_{\widehat{G^2}}}\left(\omega_i,\omega_j\right){d^2}_{{\mu }_{G^2}}\left(z_1,z_2\right))\\
&=Sc(\int_{G^4}\int_{\widehat{G^2}}{\Phi }_1({\alpha }^{-l}_i\left(\omega_i\right)){\Phi }_2({\alpha }^{-l}_j\left(\omega_j\right)){d^2}_{{\mu }_{\widehat{G^2}}}\left(\omega_i,\omega_j\right)\\
&\quad \omega_j\left(s_2\right)\omega_i\left(s_1\right)\overline{f\left(s_1,s_2\right)}f\left(z_1,z_2\right)
\overline{\omega_i\left(z_1\right)}\ \overline{\omega_j\left(z_2\right)}{d^2}_{{\mu }_{G^2}}\left(s_1,s_2\right){d^2}_{{\mu }_{G^2}}\left(z_1,z_2\right))\\
&=\int_{\widehat{G^2}}{{\left|{\mathcal{F}}_r\left(f\right)\left(\omega_i,\omega_j\right)\right|}^2{d^2}_{{\mu }_{\widehat{G^2}}}\left(\omega_i,\omega_j\right)},
\end{align*}
hence
\begin{align*}
Sc\left(\left(f*{{\rm P}}^l\right)\left(0,0\right)\right)=\int_{\widehat{G^2}}{{{\Phi }_1({\alpha }^{-l}_i\left(\omega_i\right)){\Phi }_2({\alpha }^{-l}_j\left(\omega_j\right))\left|{\mathcal{F}}_r\left(f\right)\left(\omega_i,\omega_j\right)\right|}^2{d^2}_{{\mu }_{\widehat{G^2}}}\left(\omega_i,\omega_j\right)},
\end{align*}
which completes the proof.
\end{proof}
 Using Lemmas \ref{le2.5} and \ref{le2.7}, Theorem \ref{th2.6}, and Proposition \ref{pr3.2}, we give an inversion theorem of RQFT as follows.
\begin{theorem}[ Inversion of RQFT ]\label{th3.3}
If $f\in L^1\left(G^2, \mathbb{H}\right)$, ${\mathcal{F}}_rf\in L^1\left(\widehat{G^2}, \mathbb{H}\right)$ and
\[g\left(x_1,x_2\right)=\ \int_{\widehat{G^2}}{{\mathcal{F}}_rf\left(\omega_i,\omega_j\right)\omega_j{(x}_2)\omega_i(x_1)}{d^2}_{{\mu }_{\widehat{G^2}}}\left(\omega_i,\omega_j\right),\]
then $f\left(x_1,x_2\right)=g(x_1,x_2)$ for almost every $(x_1,x_2)\in G^2$.
\end{theorem}
\begin{proof}
Form Proposition \ref{pr3.2},

\begin{align}\label{eq3.1}
&\left(f*{{\rm P}}^l\right)\left(x_1,x_2\right)=\int_{\widehat{G^2}}{\Phi }_1({\alpha }^{-l}_i\left(\omega_i\right))\cr
&{\Phi}_2({\alpha }^{-l}_j\left(\omega_j\right)){\mathcal{F}}_r\left(f\right)\left(\omega_i,\omega_j\right)
\omega_j{(x}_2)\omega_i(x_1){d^2}_{{\mu }_{\widehat{G^2}}}\left(\omega_i,\omega_j\right)
\end{align}
The integrands on the right side of \eqref{eq3.1} are bounded by $\left|{\mathcal{F}}_r\left(f\right)\left(\omega_i,\omega_j\right)\right|$, for large enough $l$. Hence, the right side of \eqref{eq3.1} converges to $g(x_1,x_2)$, for every $(x_1,x_2)\in G^2$, by the dominated convergence theorem as  ${l\to \infty }$. so by Theorem \ref{th2.6}, we get ${\mathop{\lim }_{n\to \infty } {\left\|f*P^n{\rm \ }{\rm -}f\right\|}_p=0\ }$. Thus by Lemmas (\ref{le2.5}, \ref{le2.7}), we see that ${f*P^n}$ has a pointwise convergent subsequence ${f*P^{l_n}}$ converging to $f$  almost every where. Hence $f\left(x_1,x_2\right)=g\left(x_1,x_2\right),\ $for almost every $(x_1,x_2)\in G^2$.
\end{proof}
\begin{corollary}[Uniqueness of RQFT]\label{co3.4}
If $f, g\in L^1\left(G^2, \mathbb{H}\right)$ and
\[{\mathcal{F}}_r\left(f\right)\left(\omega_i,\omega_j\right)=
{\mathcal{F}}_r\left(g\right)\left(\omega_i,\omega_j\right)\]
 for almost every $\left(\omega_i,\omega_j\right)\in \widehat{G^2}$, then
$f\left(x_1,x_2\right)=g\left(x_1,x_2\right)$, for almost every $\left(x_1,x_2\right)\in G^2$.
\end{corollary}
Now we see that, under suitable conditions, by the inverse right-sided quaternion Fourier transform (IRQFT), the original signal $f $ can be reconstructed from ${\mathcal{F}}_rf$.
\begin{definition}[IRQFT]\label{de3.5}
For every $f\in L^1\left(\widehat{G^2}, \mathbb{H}\right)$, the inverse right-sided quaternion Fourier transform of $f$ is defined by
\[\left({{\mathcal{F}}_r}^{-1}f\right)\left(x_1,x_2\right)=\int_{\widehat{G^2}}{f\left(\omega_i,\omega_j\right)\omega_j{(x}_2)\omega_i{(x}_1)}{d^2}_{{\mu }_{\widehat{G^2}}}\left(\omega_i,\omega_j\right).\]
\end{definition}
\begin{remark}\label{re3.6}
Simply, we can see that the transform ${\mathcal{F}}_r$ is bounded linear transformation from $L^1\left(G^2, \mathbb{H}\right)$ into $L^{\infty }\left(\widehat{G^2}, \mathbb{H}\right)$ and the transform ${{\mathcal{F}}_r}^{-1}$ is bounded linear transformation from $L^1\left(\widehat{G^2}, \mathbb{H}\right)$ and into $L^{\infty }\left(G^2, \mathbb{H}\right)$.
\end{remark}
We show later, ${{\mathcal{F}}_r|}_{L^1\left(G^2, \mathbb{H}\right)\cap L^2\left(G^2, \mathbb{H}\right)}$ (respectively, ${{{\mathcal{F}}_r}^{-1}|}_{L^1\left(\widehat{G^2}, \mathbb{H}\right)\cap L^2\left(\widehat{G^2}, \mathbb{H}\right)})$ can be extended to $L^2\left(G^2, \mathbb{H}\right)$ (respectively, $L^2\left(\widehat{G^2}, \mathbb{H}\right))$, and as an operator on $L^2\left(\widehat{G^2}, \mathbb{H}\right)$, ${{\mathcal{F}}_r}^{-1}$ is the inversion of ${\mathcal{F}}_r$.

An important result, the so-called  multiplication formula in classical Fourier analysis, can be generalized to RQFT. Before stating the formula, we introduce an auxiliary transform of $f\left(x_1,x_2\right)\ =\ f_0\left(x_1,x_2\right)+\ if_1\left(x_1,x_2\right)+\ jf_2\left(x_1,x_2\right)+\ kf_3\left(x_1,x_2\right)$, which is defined by
\[\beta f\left(x_1,x_2\right):=\ f_0(x_1,x_2)+ \ if_1(x_1,-x_2)+jf_2(-x_1,x_2)+kf_3(-x_1,-x_2).\]
 Then we obtain the following result.
\begin{theorem}[Modified Multiplication Formula]\label{th3.7}
Suppose that $f\in L^1\left(G^2, \mathbb{H}\right),\\\ g\in L^1\left(\widehat{G^2}, \mathbb{H}\right)$, $h:=\beta g,\ F_r:={\mathcal{F}}_rf$, and $H_r\left(x_1,x_2\right):={{\mathcal{F}}_r}^{-1}\ h\left(-x_1,-x_2\right),\ $for ($x_1,x_2)\in G^2;\ $ then
\begin{equation} \label{eq3.2}
\int_{\widehat{G^2}}{F_r\left(\omega_i,\omega_j\right)g(\omega_i,\omega_j)}{d^2}_{{\mu }_{\widehat{G^2}}}\left(\omega_i,\omega_j\right)=\int_{G^2}{f(x_1,x_2)H_r\left(x_1,x_2\right)}{d^2}_{{\mu }_{G^2}}\left(x_1,x_2\right)\ \ \
\end{equation}
Moreover, if $g$ is in $L^2\left(\widehat{G^2}, \mathbb{H}\right)$, then ${\left\|g\right\|}_2={\left\|h\right\|}_2$.
\end{theorem}
\begin{proof}
Write $g=g_0+ig_1+jg_2+kg_3$, then
\begin{align*}
&\int_{\widehat{G^2}}{\overline{\omega_i{(x}_1)}\ \overline{\omega_j{(x}_2)}g\left(\omega_i,\omega_j\right)}{d^2}_{{\mu }_{\widehat{G^2}}}\left(\omega_i,\omega_j\right)\\
&=\int_{\widehat{G^2}}{g_0\left(\omega_i,\omega_j\right)\overline{\omega_i{(x}_1)}\ \overline{\omega_j{(x}_2)}}{d^2}_{{\mu }_{\widehat{G^2}}}\left(\omega_i,\omega_j\right)\\
&+\int_{\widehat{G^2}}{ig_1\left(\omega_i,\omega_j\right)\overline{{\omega_i(x}_1)}\omega_j{(x}_2)}{d^2}_{{\mu }_{\widehat{G^2}}}\left(\omega_i,\omega_j\right)\\
&+\int_{\widehat{G^2}}{jg_2\left(\omega_i,\omega_j\right)\omega_i{(x}_1)\overline{\omega_j{(x}_2)}}{d^2}_{{\mu }_{\widehat{G^2}}}\left(\omega_i,\omega_j\right)\\
&+\int_{\widehat{G^2}}{kg_3\left(\omega_i,\omega_j\right)\omega_i{(x}_1)\omega_j{(x}_2)}{d^2}_{{\mu }_{\widehat{G^2}}}\left(\omega_i,\omega_j\right)\\
&=\int_{\widehat{G^2}}{g_0\left(\omega_i,\omega_j\right)\overline{\omega_i{(x}_1)}\ \overline{\omega_j{(x}_2)}}{d^2}_{{\mu }_{\widehat{G^2}}}\left(\omega_i,\omega_j\right)\\
&+\int_{\widehat{G^2}}{ig_1\left(\omega_i,-\omega_j\right)\overline{{\omega_i(x}_1)}\ \overline{\omega_j{(x}_2)}}{d^2}_{{\mu }_{\widehat{G^2}}}\left(\omega_i,\omega_j\right)\\
&+\int_{\widehat{G^2}}{jg_2\left(-\omega_i,\omega_j\right)\overline{\omega_i{(x}_1)}\ \overline{\omega_j{(x}_2)}}{d^2}_{{\mu }_{\widehat{G^2}}}\left(\omega_i,\omega_j\right)\\
&+\int_{\widehat{G^2}}{kg_3\left(-\omega_i,-\omega_j\right)\overline{\omega_i{(x}_1)}\ \overline{\omega_j{(x}_2)}}{d^2}_{{\mu }_{\widehat{G^2}}}\left(\omega_i,\omega_j\right)\\
&=\int_{\widehat{G^2}}{\beta g\left(\omega_i,\omega_j\right)\overline{\omega_i{(x}_1)}\ \overline{\omega_j{(x}_2)}}{d^2}_{{\mu }_{\widehat{G^2}}}\left(\omega_i,\omega_j\right)\\
&=\int_{\widehat{G^2}}{h\left(\omega_i,\omega_j\right)\overline{\omega_i{(x}_1)}\ \overline{\omega_j{(x}_2)}}{d^2}_{{\mu }_{\widehat{G^2}}}\left(\omega_i,\omega_j\right)\\
&=\int_{\widehat{G^2}}{h\left(\omega_i,\omega_j\right)\omega_i{(-x}_1)\omega_j{(-x}_2)}{d^2}_{{\mu }_{\widehat{G^2}}}\left(\omega_i,\omega_j\right)\\
&=H_r\left(x_1,x_2\right).
\end{align*}
Applying Fubini's theorem, we get
\begin{align*}
&\int_{\widehat{G^2}}{{\mathcal{F}}_r(f)\left(\omega_i,\omega_j\right)g(\omega_i,\omega_j)}{d^2}_{{\mu }_{\widehat{G^2}}}\left(\omega_i,\omega_j\right)\\
&=\int_{\widehat{G^2}}{\left(\int_{G^2}{f\left(x_1,x_2\right)}\overline{\omega_i{(x}_1)}\ \overline{\omega_j{(x}_2)}{d^2}_{{\mu }_{G^2}}\left(x_1,x_2\right)\right)g\left(\omega_i,\omega_j\right)}{d^2}_{{\mu }_{\widehat{G^2}}}\left(\omega_i,\omega_j\right)\\
&=\int_{G^2}{f\left(x_1,x_2\right)\left(\int_{\widehat{G^2}}{\overline{\omega_i{(x}_1)}\ \overline{\omega_j{(x}_2)}g\left(\omega_i,\omega_j\right)}{d^2}_{{\mu }_{\widehat{G^2}}}\left(\omega_i,\omega_j\right)\right)}{d^2}_{{\mu }_{G^2}}\left(x_1,x_2\right)\\
&=\int_{G^2}{f\left(x_1,x_2\right)H_r\left(x_1,x_2\right)}{d^2}_{{\mu }_{G^2}}\left(x_1,x_2\right).
\end{align*}
If $g\ \in \ L^2\left(\widehat{G^2}, \mathbb{H}\right)$, then it is easy to verify that ${\left\|g\right\|}_2={\left\|h\right\|}_2$ by the definition of $\beta $.
\end{proof}
\begin{remark}\label{re3.8}
The multiplication formula of complex Fourier transform has the form\\
 \[\int_{\widehat{G^2}}{\mathcal{F}}_s\left(f\right)\left(\omega_i,\omega_j\right){\mathcal{F}}_s\left(g\right)\left(\omega_i,\omega_j\right){d^2}_{{\mu }_{\widehat{G^2}}}\left(\omega_i,\omega_j\right)=\int_{G^2}{f(x_1,x_2)g\left(x_1,x_2\right)}{d^2}_{{\mu }_{G^2}}\left(x_1,x_2\right).\]\\
But when $f\in L^1\left(G^2, \mathbb{H}\right)\ $and $g\in L^1\left(\widehat{G^2}, \mathbb{H}\right)$, this standard formula is not valid for RQFT of integrable functions. Using the auxiliary transform $\beta $, we may obtain an analogous formula \eqref{eq3.2} for quaternion-valued integrable functions.
\end{remark}
\subsection{The Plancherel theorem of RQFT} \label{su3.2}

Based on complex version of the Plancherel theorem, we are going to show that if $f\in L^1\left(G^2, \mathbb{H}\right)\cap L^2\left(G^2, \mathbb{H}\right)$, then it turns out that $\hat{f}\in L^2\left(\widehat{G^2}, \mathbb{H}\right)$ and $\|\hat{f}\|_2=\|f\|_2$, where $\hat{f}$ is the RFT of $f$. Moreover, this isometry of $L^1\left(G^2, \mathbb{H}\right)\cap L^2\left(G^2, \mathbb{H}\right)$ into $L^2\left(G^2, \mathbb{H}\right)$ extends to an isometric of $L^2\left(G^2, \mathbb{H}\right)$ onto $L^2\left(G^2, \mathbb{H}\right)$, and this extension defines the Fourier transform of every $f\in \ L^2\left(G^2, \mathbb{H}\right)$. The convolution theorem plays a vital role in proving the Plancherel theorem. However, the classical convolution theorem no longer holds for the QFT. The Plancherel theorem of QFT was discussed in recent research papers (see \cite{Cheng,Hart}) in the case $G=\mathbb{R}$. We give a restatement of the Plancherel theorem here, since the prerequisites for setting up of the theorem may not be put forward, so clearly in recent research papers. It is probably worth pointing out that Proposition 3.2 plays a key role in our proof.
\begin{theorem}\label{th3.9}
If $f\in\ L^1\left(G^2, \mathbb{H}\right)\cap L^2\left(G^2, \mathbb{H}\right)$, then ${\mathcal{F}}_rf\in\ L^2\left(\widehat{G^2}, \mathbb{H}\right)$ and Parseval's identity ${\left\|{\mathcal{F}}_rf\right\|}^2_2={\left\|f\right\|}^2_2$ holds.
\end{theorem}
\begin{proof}
We fix $f\in \ L^1\left(G^2, \mathbb{H}\right)\cap L^2\left(G^2, \mathbb{H}\right)$. Put $\tilde{f}\left(x_1,x_2\right):=\overline{f\left(-x_1,{-x}_2\right)}$ and define $g\left(x_1,x_2\right)=\left(\tilde{f}*f\right)\left(x_1,x_2\right)$. Trivially
\[g\left(x_1,x_2\right)=\int_{G^2}{\overline{f\left(y_1,y_2\right)}f\left(x_1+y_1,x_2+y_2\right)}\ {d^2}_{{\mu }_{G^2}}\left(y_1,y_2\right).\]
 Since by Lemma \ref{le2.4} $\left(x_1,x_2\right)\mapsto L_{(-x_1,{-x}_2)}\overline{f}$ is a continuous mapping of $G^2$ into $L^2\left(G^2, \mathbb{H}\right)$ and by the continuity of the inner product, we see that $g\left(x_1,x_2\right)$ is a continuous function.\\
The function $g$ is bounded by the Cauchy-Schwartz inequality;
\[\left|g\left(x_1,x_2\right)\right|\le {\left\|L_{(-x_1,{-x}_2)}\overline{f}\right\|}_2{\left\|f\right\|}_2={\left\|f\right\|}^2_2.\]
  Furthermore, $g\in \ L^1\left(G^2, \mathbb{H}\right)$, since $f\in \ L^1\left(G^2, \mathbb{H}\right)$ and $\hat{f}\in \ L^1\left(\widehat{G^2}, \mathbb{H}\right)$. Moreover $g$ is continuous and bounded and Lemma \ref{le2.7} shows that
\[{\mathop{\lim }_{l\to \infty } Sc\left(\left(g*{{\rm p}}^l\right)\left(0,0\right)\right)=Sc\left(g\left(0,0\right)\right)={\left\|f\right\|}^2_2.\ }.\]
On the other hand, since $g\in \ L^1\left(G^2, \mathbb{H}\right)$, by Proposition \ref{pr3.2} we have
\[Sc\left(\left(g*{{\rm P}}^l\right)\left(0,0\right)\right)=\int_{\widehat{G^2}}{{\Phi }_1({\alpha }^{-l}_i\left(\omega_i\right)){\Phi }_2({\alpha }^{-l}_j\left(\omega_j\right)){\left|{\mathcal{F}}_r\left(f\right)\left(\omega_i,\omega_j\right)\right|}^2{d^2}_{{\mu }_{\widehat{G^2}}}\left(\omega_i,\omega_j\right)}.
\]
Since $0\le {\Phi }_1({\alpha }^{-l}_i\left(\omega_i\right)){\Phi }_2({\alpha }^{-l}_j\left(\omega_j\right)){\left|{\mathcal{F}}_r\left(f\right)\left(\omega_i,\omega_j\right)\right|}^2$ increases to ${\left|{\mathcal{F}}_r\left(f\right)\left(\omega_i,\omega_j\right)\right|}^2$ as $l \to \infty$,  the dominated convergence theorem gives
\[{\mathop{\lim }_{l{\rm \ }\to {\rm \ }\infty } Sc\left(\left(g*{{\rm P}}^l\right)\left(0,0\right)\right)\ }=\int_{\widehat{G^2}}{{\left|{\mathcal{F}}_r\left(f\right)\left(\omega_i,\omega_j\right)\right|}^2{d^2}_{{\mu }_{\widehat{G^2}}}\left(\omega_i,\omega_j\right)}={\left\|{\mathcal{F}}_rf\right\|}^2_2.\]
Therefore, ${\mathcal{F}}_rf\in \ L^2\left(\widehat{G^2}, \mathbb{H}\right)$ and ${\left\|{\mathcal{F}}_rf\right\|}^2_2={\left\|f\right\|}^2_2$.
\end{proof}

By Theorem \ref{th3.9}, ${{\mathcal{F}}_r|}_{L^1\left(G^2, \mathbb{H}\right)\cap L^2\left(G^2, \mathbb{H}\right)}$ is an isometry of $L^1\left(G^2, \mathbb{H}\right)\cap L^2\left(G^2, \mathbb{H}\right)$ into $L^2\left(\widehat{G^2}, \mathbb{H}\right)$. Since $L^1\left(G^2, \mathbb{H}\right)\cap L^2\left(G^2, \mathbb{H}\right)$ is a dense subset of $L^2\left(G^2, \mathbb{H}\right)$, therefore, there exists a unique bounded  extension, say $\Omega_r$, of ${{\mathcal{F}}_r|}_{L^1\left(G^2, \mathbb{H}\right)\cap L^2\left(G^2, \mathbb{H}\right)}$ to all of $L^2\left(G^2, \mathbb{H}\right)$. If $f\in L^2\left(G^2, \mathbb{H}\right)\ $ and $F_r=\Omega_rf$ is defined by the $L^2$-limit of the sequence \{${{\mathcal{F}}_rf_l\}}_{l\in \mathbb{N}}$, where ${\{f_l\}}_{l\in \mathbb{N}}$ is any sequence in $L^1\left(G^2, \mathbb{H}\right)\cap L^2\left(G^2, \mathbb{H}\right)$ converging to $f$ in the $L^2$-norm. If $f_l\left(x_1,x_2\right)=f\left(x_1,x_2\right){\chi }_{\bigcup_{n=1}^l}{{\alpha }^{-n}(K_0)}\left(x_1,x_2\right)$, where $K_0$ is a compact symmetric Neighborhood of identity in $G^2$ and $\alpha \ $ is as in Lemma 2.4, then
\[F_r\left(\omega_i,\omega_j\right)={\mathop{\lim }_{l{\rm \ }\to \infty } \int_{\bigcup_{n=1}^l }{{\alpha }^{-n}(K_0)}}{{\rm \ }f\left(x_1,x_2\right)\overline{\omega_i(x_1)}\ \overline{\omega_j(x_2)}{d^2}_{{\mu }_{G^2}}\left(x_1,x_2\right)},\]
where $f=\mathop{{\rm lim}}_{l\to \infty }f_l$ means ${\left\|f-f_l\right\|}_2\to 0$ as $l \to \infty$.

We call $F_r=\Omega_rf$  the RQFT from $L^2\left(G^2, \mathbb{H}\right)$ into $L^2\left(\widehat{G^2}, \mathbb{H}\right)$.The multiplication formula \eqref{eq3.2} easily extends to $L^2\left(\widehat{G^2}, \mathbb{H}\right)$. The left $\mathbb{H}$-linear operator $\Omega_r\ $on $L^2\left(\widehat{G^2}, \mathbb{H}\right)\ $is an isometry. So $\Omega_r\ $is a one-to-one mapping. Moreover, we can show that $\Omega_r\ $is onto.
\begin{theorem}\label{th3.10}
The RQFT, $\Omega_r$, is a unitary operator from $L^2\left(G^2, \mathbb{H}\right)$ onto $L^2\left(\widehat{G^2}, \mathbb{H}\right)$.
\end{theorem}
\begin{proof}
Firstly, we show that the range of $\Omega_r$, denoted by $R( \Omega_r)$, is a closed subspace of $L^2\left(\widehat{G^2}, \mathbb{H}\right)$. Let ${F_{r_l}:=\Omega_r\left(f_l\right)}$, ${l\in \mathbb{N}}$, be a sequence in $R(\Omega_r)$ converging to ${{\rm F}}$ in $L^2$-norm sense. The isometric property shows that $\Omega_r\ $ is continuous and$\ {\{f_l\}}_{l\in \mathbb{N}}$ is also a Cauchy sequence. The completeness of $L^2\left(G^2, \mathbb{H}\right)$ implies that $\ {\{f_l\}}_{l\in \mathbb{N}}$ converges to some $f\in L^2\left(G^2, \mathbb{H}\right)$, and the continuity of $\Omega_r\ $shows that
\[\Omega_rf={\mathop{\lim }_{l\to \infty } {\rm \ }\Omega_r\left(f_l\right)\ }={\rm F}.\]
 If $R(\Omega_r)$ is not all of $L^2\left(\widehat{G^2}, \mathbb{H}\right)$, as every closed subspace of the Hilbert space $L^2\left(G^2, \mathbb{H}\right)\ $has an orthogonal complement, we could find a function $ u$ such that \[\int_{\widehat{G^2}}{\Omega_r(f)\left(\omega_i,\omega_j\right)\overline{u}(\omega_i,\omega_j)}{d^2}_{{\mu }_{\widehat{G^2}}}\left(\omega_i,\omega_j\right)=0\]
for all $f\ \in L^1\left(G^2, \mathbb{H}\right)$ and ${\left\|u\right\|}_2\ne 0$. Let $g=\overline{u},h=\beta g$; then by multiplication formula,
\[\int_{\widehat{G^2}}{F_r\left(\omega_i,\omega_j\right)g(\omega_i,\omega_j)}{d^2}_{{\mu }_{\widehat{G^2}}}\left(\omega_i,\omega_j\right)=\int_{G^2}{f(x_1,x_2)H_r\left(x_1,x_2\right)}{d^2}_{{\mu }_{G^2}}\left(x_1,x_2\right)=0\]
 for all $f\ \in L^2\left(G^2, \mathbb{H}\right)\cap L^2\left(G^2, \mathbb{H}\right)$. Pick$\ f=\overline{H_r}$, this implies that $H_r\left(x_1,x_2\right)=0$ for almost every $(x_1,x_2)\in G^2$, contradicting the fact that ${\left\|H_r\right\|}_2={\left\|h\right\|}_2={\left\|g\right\|}_2={\left\|u\right\|}_2\ne 0$.
 \end{proof}
Next result shows that the mapping $\Omega_r$ is a Hilbert space isomorphism of $L^2\left(\widehat{G^2}, \mathbb{H}\right)$, that is, preserving inner product or so-called the Parseval theorem.
\begin{theorem}\label{thkh3.10}
Let $f,g\in \ L^2\left(G^2, \mathbb{H}\right)$ and $F_r=\Omega_rf$, ${\Gamma }_r=\Omega_rg$. Then
\[\int_{G^2}{f\left(x_1,x_2\right)\overline{g\left(x_1,x_2\right)}{d^2}_{{\mu }_{G^2}}\left(x_1,x_2\right)}=\int_{\widehat{G^2}}{F_r\left(\omega_i,\omega_j\right)\overline{{\Gamma }_r\left(\omega_i,\omega_j\right)}{d^2}_{{\mu }_{\widehat{G^2}}}\left(\omega_i,\omega_j\right)}.\]
\end{theorem}
\begin{proof}
Let
\[{{\rm p}}_0+i{{\rm p}}_1+j{{\rm p}}_2+k{{\rm p}}_3=\int_{G^2}{f\left(x_1,x_2\right)}\overline{g\left(x_1,x_2\right)}{d^2}_{{\mu }_{G^2}}\left(x_1,x_2\right)\]
and
\[{{\rm q}}_0+i{{\rm q}}_1+j{{\rm q}}_2+k{{\rm q}}_3=\int_{\widehat{G^2}}{F_rf\left(\omega_i,\omega_j\right)}\overline{{{\rm \ }\Gamma }_r\left(\omega_i,\omega_j\right)}{d^2}_{{\mu }_{\widehat{G^2}}}\left(\omega_i,\omega_j\right).\]
From the Parseval's identity, we have
\[{\left\|f+g\right\|}^2_2={\left\|f\right\|}^2_2+{\left\|g\right\|}^2_2+2{{\rm p}}_0={\left\|F_r+{{\rm \ }\Gamma }_r\right\|}^2_2={\left\|F_r\right\|}^2_2+{\left\|{{\rm \ }\Gamma }_r\right\|}^2_2+2{{\rm q}}_0.\]
 Thus ${{\rm p}}_0={{\rm q}}_0$. By using properties of $L^2$-norm and applying Parseval's identity to the equalities
 ${\left\|f+ig\right\|}^2_2={\left\|F_r+{{\rm \ i}\Gamma }_r\right\|}^2_2$, ${\left\|f+jg\right\|}^2_2={\left\|F_r+{{\rm \ j}\Gamma }_r\right\|}^2_2$  and ${\left\|f+kg\right\|}^2_2={\left\|F_r+{{\rm \ k}\Gamma }_r\right\|}^2_2$,
 respectively, we can get ${{\rm p}}_m={{\rm q}}_m , (m = 1,2,3)$, which completes the proof.
\end{proof}
\begin{theorem}\label{th3.12}
The inverse $f={{{\rm \ }\Omega }_r}^{-1}F_r$ is the $L^2$-limit of the sequence \\
${\left\{{{\mathcal{F}}_r}^{-1}F_{r_l}\right\}}_{l\in \mathbb{N}}\ $, where ${\left\{F_{r_l}\right\}}_{l\in \mathbb{N}}\ $is any sequence in $L^1\left(\widehat{G^2}, \mathbb{H}\right)\cap L^2\left(\widehat{G^2}, \mathbb{H}\right)$\\ converging to $F_r\ $in the $L^2\left(\widehat{G^2}, \mathbb{H}\right)$ norm. If $F_{r_l}=F_r{\chi }_{\bigcup_{n=1}^l}{{\alpha }^{-l}(K_0)}$, where $K_0$ is a compact neighborhood of identity in $\widehat{G^2}$ and $\alpha $ is an automorphism in $\widehat{G^2}$, then
\[f\left(x_1,x_2\right)={\mathop{\lim }_{l\to \infty } \int_{\bigcup_{n=1}^l{{\alpha }^{-n}(K_0)}}{F_r\left(\omega_i,\omega_j\right)\omega_j(x_2)\omega_i(x_1){d^2}_{{\mu }_{\widehat{G^2}}}\left(\omega_i,\omega_j\right)}\ }.\]
In particular, if $F_r\in L^1\left(\widehat{G^2}, \mathbb{H}\right)\cap L^2\left(\widehat{G^2}, \mathbb{H}\right)$, then
\[f\left(x_1,x_2\right)=\int_{\widehat{G^2}}{F_r\left(\omega_i,\omega_j\right)\omega_j(x_2)\omega_i(x_1){d^2}_{{\mu }_{\widehat{G^2}}}\left(\omega_i,\omega_j\right)}.\]
\end{theorem}
\begin{proof}
The quaternionic Riesz representation theorem (see \cite{Ghil}) guarantees that there exists a unique operator $\Omega^*_r \in B\left(L^2\left(\widehat{G^2}, \mathbb{H}\right),L^2\left(G^2, \mathbb{H}\right)\right)$, which is called the adjoint of $\Omega_r$, such that for all $f\in \ L^2\left(G^2, \mathbb{H}\right)$ and $g\in \ L^2\left(\widehat{G^2}, \mathbb{H}\right)$, $\left(\Omega_rf,g\right)=(f,\Omega^*_rg)$. Since $\Omega_r$is unitary, then ${\Omega_r}^{-1}=\Omega^*_r$. For any fixed $F_r\in L^2\left(\widehat{G^2}, \mathbb{H}\right)$, let ${\left\{F_{r_l}\right\}}_{l\in \mathbb{N}}$ be an arbitrary sequence in $L^1\left(\widehat{G^2}, \mathbb{H}\right)\cap L^2\left(\widehat{G^2}, \mathbb{H}\right)$ converging to $F_r$ in the $L^2\left(\widehat{G^2}, \mathbb{H}\right)$norm; then
\begin{align*}
&(g,\Omega^*_r F_r)=\left({\Gamma }_r,F_r\right)={\mathop{\lim }_{l\to \infty } \left({\Gamma }_r,F_{r_l}\right)\ }\\
&={\mathop{\lim }_{l\to \infty } \int_{\widehat{G^2}}{(\int_{\widehat{G^2}}{g\left(\omega_i,\omega_j\right)\overline{\omega_i\left(x_1\right)}\ \overline{\omega_j\left(x_2\right)}{d^2}_{{\mu }_{G^2}}\left(\omega_i,\omega_j\right))\overline{F_{r_l}\left(\omega_i,\omega_j\right)}{d^2}_{{\mu }_{\widehat{G^2}}}\left(\omega_i,\omega_j\right)}}\ }\\
&={\mathop{\lim }_{l\to \infty } \int_{\widehat{G^2}}{g\left(\omega_i,\omega_j\right)(\int_{\widehat{G^2}}{\overline{F_{r_l}\ \left(\omega_i,\omega_j\right)} \ \overline{\omega_i\left(x_1\right)}\ \overline{\omega_j\left(x_2\right)}{d^2}_{{\mu }_{\widehat{G^2}}}\left(\omega_i,\omega_j\right))}}\ }{d^2}_{{\mu }_{\widehat{G^2}}}\left(\omega_i,\omega_j\right)\\
&={\mathop{\lim }_{l\to \infty } \left\langle g,{{\mathcal{F}}_r}^{-1}F_{r_l}\right\rangle =\left\langle g,{\mathop{\lim }_{l\to \infty } {{\mathcal{F}}_r}^{-1}F_{r_l}\ }\right\rangle \ }.\end{align*}
For all $g\in L^1\left(\widehat{G^2} , \mathbb{H}\right)\cap L^2\left(\widehat{G^2}, \mathbb{H}\right)$. Thus $f={\Omega_r}^{-1}F_r=\Omega^*_rF_r={\mathop{\lim }_{l\to \infty } {{\mathcal{F}}_r}^{-1}F_{r_l}\ }$

In particular, if $F_r \in L^1\left(\widehat{G^2}, \mathbb{H}\right)\cap L^2\left(\widehat{G^2}, \mathbb{H}\right)$, then
\[f\left(x_1,x_2\right)=\int_{\widehat{G^2}}{F_r\left(\omega_i,\omega_j\right)\omega_j\left(x_2\right)\omega_i\left(x_1\right){d^2}_{{\mu }_{\widehat{G^2}}}\left(\omega_i,\omega_j\right),\ }\]
which completes the proof.
\end{proof}
\section{ The two-sided quaternion Fourier transform on locally compact abelian group}\label{se4}

In this section, based on the proof of the inversion formula for right sided quaternionic Fourier transform, we are going to prove the inversion formula for two-sided (sandwich) quaternion Fourier transform (SQFT).
\subsection{ The two-sided quaternion Fourier transform  in $L^2\left(G^2, \mathbb{H}\right)$ }\label{su4.1}
We have seen that the right sided quaternionic Fourier transform lacks some needed properties. Therefore we are going to sandwich the function in between the two Fourier characters, in order to obtain some more symmetric features.
The two-sided quaternionic Fourier transform (SQFT) of $f\in L^2({\mathbb{R}}^2, \mathbb{H})$ is considered in \cite{Cheng,Ell1,Hart,Bas,Bulow,Ell4}. By a similar argument we may define two-sided quaternionic Fourier transform (SQFT) of $f\in L^2\left(G^2 , \mathbb{H}\right)$, which is a function from $\widehat{G^2}\ $ to $\mathbb{H}$ as follows:
\[{\mathcal{F}}_s\left(f\right)\left(\omega_i,\omega_j\right)=\int_{G^2}{\overline{\omega_i(x_1)}\ f\left(x_1,x_2\right)\overline{\omega_j{(x}_2)}}\ {d^2}_{{\mu }_{G^2}}\left(x_1,x_2\right).\]
Unlike the RQFT, SQFT is not a left $\mathbb{H}$-linear operator. But SQFT is left ${\mathbb{C}}_i$-linear and right ${\mathbb{C}}_j$-linear. Moreover, SQFT could establish relationship with RQFT through the following transform.
\begin{definition}\label{de4.2}
For any function $f\left(x_1,x_2\right)=f_0\left(x_1,x_2\right)+if_1{\left(x_1,x_2\right)
+jf}_2\left(x_1,x_2\right)+kf_3\left(x_1,x_2\right)$, we define the transform $\mathcal{W}$ of $f$ by
\[\mathcal{W}f\left(x_1,x_2\right):=f_0\left(x_1,x_2\right)+if_1{\left(x_1,x_2\right)+jf}_2\left(-x_1,x_2\right)+kf_3\left(-x_1,x_2\right).\]
One can simply see that the transform $\mathcal{W}$ is bijection mapping on $L^p\left(G^2, \mathbb{H}\right)$, $p =1,2$. So the inverse of $\mathcal{W}$ can be defined and ${\mathcal{W}}^{-1}$ is actually equal to $\mathcal{W}$ itself.
\end{definition}
\begin{proposition}\label{pr4.3}
Let $f,g\in L^p\left(G^2, \mathbb{H}\right)$, $p =1,2$. Then the following assertions hold:
\begin{enumerate}
\item[(i)]
The transform $\mathcal{W}$ is a left ${\mathbb{C}}_i$-linear mapping on $L^p\left(G^2, \mathbb{H}\right)$.
\item[(ii)]
If $f\in L^p\left(G^2, \mathbb{H}\right)$, then
\begin{equation}
{\mathcal{F}}_sf={\mathcal{F}}_r(\mathcal{W}f). \label{eq4.1}
\end{equation}
Moreover, if $f\ $is ${\mathbb{C}}_i$-valued or $f\ $ is even with respect to first variable, then $ \mathcal{F}_sf={\mathcal{F}}_rf.$
\item[(iii)]
If $f,g\in L^2\left(G^2, \mathbb{H}\right)$, then
\[\left\langle \mathcal{W}f,\mathcal{W}g\right\rangle =\left\langle f,g\right\rangle, \ Sc\left(i\left(f,g\right)\right)=Sc(i\left(\mathcal{W}f,\mathcal{W}g\right)).\]
In particular ${\left\|f\right\|}_2={\left\|\mathcal{W}f\right\|}_2$.
\end{enumerate}
\end{proposition}
\begin{proof}
(i)
From the definition of $\mathcal{W}$, we get
\begin{align*}
&\mathcal{W}\left(if\right)\left(x_1,x_2\right)=if_0\left(x_1,x_2\right)+iif_1{\left(x_1,x_2\right)+ijf}_2\left(-x_1,x_2\right)+ikf_3\left(-x_1,x_2\right)\\
&=i\left(f_0\left(x_1,x_2\right)+if_1{\left(x_1,x_2\right)+jf}_2\left(-x_1,x_2\right)+kf_3\left(-x_1,x_2\right)\right)\\
&=i\mathcal{W}\left(f\right)\left(x_1,x_2\right)
\end{align*}
Then $\mathcal{W}$ is a ${\mathbb{C}}_i$-linear transform.\\
To prove (ii), let $h:=\mathcal{W}f$. Then
\begin{align*}
{\mathcal{F}}_s\left(f\right)\left(\omega_i,\omega_j\right)&=\int_{G^2}{\overline{\omega_i(x_1)}\ f\left(x_1,x_2\right)\overline{\omega_j{(x}_2)}}\ {d^2}_{{\mu }_{G^2}}\left(x_1,x_2\right)\\
&=\int_{G^2}{\ f_0\left(x_1,x_2\right)\overline{\omega_i(x_1)}\ \overline{\omega_j{(x}_2)}}\ {d^2}_{{\mu }_{G^2}}\left(x_1,x_2\right)\\
&+\int_{G^2}{\ if_1\left(x_1,x_2\right)\overline{\omega_i(x_1)}\ \overline{\omega_j{(x}_2)}}\ {d^2}_{{\mu }_{G^2}}\left(x_1,x_2\right)\\
&+\int_{G^2}{j\ f_2\left(x_1,x_2\right)\omega_i(x_1)\overline{\omega_j{(x}_2)}}\ {d^2}_{{\mu }_{G^2}}\left(x_1,x_2\right)\\
&+\int_{G^2}{k\ f_3\left(x_1,x_2\right)\omega_i(x_1)\overline{\omega_j{(x}_2)}}\ {d^2}_{{\mu }_{G^2}}\left(x_1,x_2\right)\\
&=\int_{G^2}{\ f_0\left(x_1,x_2\right)\overline{\omega_i(x_1)}\ \overline{\omega_j{(x}_2)}}\ {d^2}_{{\mu }_{G^2}}\left(x_1,x_2\right)\\
&+\int_{G^2}{\ if_1\left(x_1,x_2\right)\overline{\omega_i(x_1)}\ \overline{\omega_j{(x}_2)}}\ {d^2}_{{\mu }_{G^2}}\left(x_1,x_2\right)\\
&+\int_{G^2}{j\ f_2\left(-x_1,x_2\right)\overline{\omega_i\left(x_1\right)}\ \overline{\omega_j{(x}_2)}}\ {d^2}_{{\mu }_{G^2}}\left(x_1,x_2\right)\\
&+\int_{G^2}{k\ f_3\left({-x}_1,x_2\right)\overline{\omega_i(x_1)}\ \overline{\omega_j{(x}_2)}}\ {d^2}_{{\mu }_{G^2}}\left(x_1,x_2\right)\\
&=\int_{G^2}{\ h\left(x_1,x_2\right)\overline{\omega_i(x_1)}\ \overline{\omega_j{(x}_2)}}\ {d^2}_{{\mu }_{G^2}}\left(x_1,x_2\right)={\mathcal{F}}_rh\left(\omega_i,\omega_j\right).
\end{align*}
If $f\ $ is ${\mathbb{C}}_i$-valued or $f$ is an even function with respect to first variable, then $\mathcal{W}f=f$. It follows that ${\mathcal{F}}_s f={\mathcal{F}}_r(\mathcal{W}f)$.\\
Now we prove (ii). If $f,g\in L^2\left(G^2, \mathbb{H}\right),\ $ then
\begin{align*}
\left\langle \mathcal{W}f,\mathcal{W}g\right\rangle &=Sc(\int_{G^2}{\ f_0\left(x_1,x_2\right)g_0\left(x_1,x_2\right)}\ {d^2}_{{\mu }_{G^2}}\left(x_1,x_2\right)\\
&+\int_{G^2}{f_1\left(x_1,x_2\right)g_1\left(x_1,x_2\right)}\ {d^2}_{{\mu }_{G^2}}\left(x_1,x_2\right)\\
&+\int_{G^2}{\ f_2\left(-x_1,x_2\right)g_2\left(-x_1,x_2\right)}\ {d^2}_{{\mu }_{G^2}}\left(x_1,x_2\right)\\
&+\int_{G^2}{\ f_3\left({-x}_1,x_2\right)\ g_3\left({-x}_1,x_2\right)}\ {d^2}_{{\mu }_{G^2}}\left(x_1,x_2\right))\\
&=Sc(\int_{G^2}{\ f_0\left(x_1,x_2\right)g_0\left(x_1,x_2\right)}\ {d^2}_{{\mu }_{G^2}}\left(x_1,x_2\right)\\
&+\int_{G^2}{f_1\left(x_1,x_2\right)g_1\left(x_1,x_2\right)}\ {d^2}_{{\mu }_{G^2}}\left(x_1,x_2\right)\\
&+\int_{G^2}{\ f_2\left(x_1,x_2\right)g_2\left(x_1,x_2\right)}\ {d^2}_{{\mu }_{G^2}}\left(x_1,x_2\right)\\
\end{align*}
\begin{align*}
&+\int_{G^2}{\ f_3\left(x_1,x_2\right)\ g_3\left(x_1,x_2\right)}\ {d^2}_{{\mu }_{G^2}}\left(x_1,x_2\right))\\
&=\left\langle f,g\right\rangle.
\end{align*}
Since $\mathcal{W}$ is left ${\mathbb{C}}_i$-linear, then
\begin{align*}
& Sc(i\int_{G^2}{f\left(x_1,x_2\right)\overline{g\left(x_1,x_2\right)}}\ {d^2}_{{\mu }_{G^2}}\left(x_1,x_2\right))\\
&=Sc(\int_{G^2}{if\left(x_1,x_2\right)\overline{g\left(x_1,x_2\right)}}\ {d^2}_{{\mu }_{G^2}}\left(x_1,x_2\right))\\
&=Sc\left(\int_{G^2}{\mathcal{W}(if)\left(x_1,x_2\right)\overline{\mathcal{W}(g)\left(x_1,x_2\right)}}\ {d^2}_{{\mu }_{G^2}}\left(x_1,x_2\right)\right)\\
&=Sc(i\left(\left(\mathcal{W}f,\mathcal{W}g\right)\right).
\end{align*}
Finally
\[{\left\|\mathcal{W}f\right\|}^2_2={\left\|f_0\right\|}^2_2{+\left\|f_1\right\|}^2_2{+\left\|f_2\right\|}^2_2{+\left\|f_3\right\|}^2_2={\left\|f\right\|}^2_2,\]
which completes the proof. \end{proof}
\begin{theorem}[Inversion of SQFT]\label{th4.4}
If $f\in L^1\left(G^2, \mathbb{H}\right),\ {\mathcal{F}}_sf\in L^1\left(\widehat{G^2}, \mathbb{H}\right)$ and
\begin{align}\label{eq4.2}
g\left(x_1,x_2\right)=\int_{\widehat{G^2}}{\omega_i(x_1){\mathcal{F}}_sf\left(\omega_i,\omega_j\right)\omega_j{(x}_2)}{d^2}_{{\mu }_{\widehat{G^2}}}\left(\omega_i,\omega_j\right),
\end{align}

then $f\left(x_1,x_2\right)=g(x_1,x_2)$ for almost every $(x_1,x_2)\in G^2$.
\end{theorem}
\begin{proof}
By Proposition \ref{pr4.3} (ii), we have ${\mathcal{F}}_s f={\mathcal{F}}_r(\mathcal{W}f)$. Let $h:={{\mathcal{F}}_r}^{-1}\left({\mathcal{F}}_sf\right)$, then by Theorem \ref{th3.3}, \[h\left(x_1,x_2\right)={{\mathcal{F}}_r}^{-1}({\mathcal{F}}_r\mathcal{W}f)\left(x_1,x_2\right)
=(\mathcal{W}f)\left(x_1,x_2\right)\] for almost every $\left(x_1,x_2\right)\in G^2$.
To prove $f\left(x_1,x_2\right)=g(x_1,x_2)$, for almost every $(x_1,x_2)\in G^2$, it is enough to verify $\mathcal{W}g=h$. Note that
\[h\left(x_1,x_2\right)=\int_{\widehat{G^2}}{{\mathcal{F}}_sf\left(\omega_i,\omega_j\right)\omega_j{(x}_2)\omega_i(x_1)}{d^2}_{{\mu }_{\widehat{G^2}}}\left(\omega_i,\omega_j\right).\]
From \eqref{eq4.2}, we can see that $Sc\left(h\left(x_1,x_2\right)\right)=Sc(g\left(x_1,x_2\right))$.

 Since
\begin{align*}
h\left(x_1,x_2\right)j&=\int_{\widehat{G^2}}
{{\mathcal{F}}_sf\left(\omega_i,\omega_j\right)\omega_j{(x}_2)\omega_i(x_1)}{d^2}_{{\mu }_{\widehat{G^2}}}\left(\omega_i,\omega_j\right)j\\
&=\int_{\widehat{G^2}}{{\mathcal{F}}_sf\left(\omega_i,\omega_j\right)\omega_j{(x}_2)j\overline{\omega_i(x_1)}}{d^2}_{{\mu }_{\widehat{G^2}}}\left(\omega_i,\omega_j\right).
\end{align*}
Then
\begin{align*}
Sc\left(h\left(x_1,x_2\right)j\right)&=Sc(\int_{\widehat{G^2}}{{\mathcal{F}}_sf\left(\omega_i,\omega_j\right)\omega_j{(x}_2)j\overline{\omega_i\left(x_1\right)}}{d^2}_{{\mu }_{\widehat{G^2}}}\left(\omega_i,\omega_j\right))\\
&=Sc(\int_{\widehat{G^2}}{\overline{\omega_i\left(x_1\right)}{\mathcal{F}}_sf\left(\omega_i,\omega_j\right)\omega_j{(x}_2)j}{d^2}_{{\mu }_{\widehat{G^2}}}\left(\omega_i,\omega_j\right))\\
&=Sc\left(g\left(-x_1,x_2\right)j\right).
\end{align*}
Similarly we have
 \[Sc\left(h\left(x_1,x_2\right)i\right)=Sc(g\left(x_1,x_2\right)i) \mbox{ and } Sc\left(h\left(x_1,x_2\right)k\right)=Sc\left(g\left({-x}_1,x_2\right)k\right).\]
 Hence we conclude that $\mathcal{W}g=h$.
\end{proof}
So we can define the inverse two-sided quaternion Fourier transform by \eqref{eq4.2} or equivalently by ${\mathcal{W}}^{-1}{{\mathcal{F}}_r}^{-1}{\mathcal{F}}_s$.

\begin{definition}[ISQFT]\label{de4.5}
For every $f\in L^1\left(G^2, \mathbb{H}\right)$, the inverse two-sided quaternion Fourier transform of ${\mathcal{F}}_s$ is defined by
\[\left({{\mathcal{F}}_s}^{-1}f\right)\left(x_1,x_2\right)=\int_{\widehat{G^2}}{\omega_i{(x}_1)f\left({\omega}_i,\omega_j\right)\omega_j{(x}_2)}{d^2}_{{\mu }_{\widehat{G^2}}}\left(\omega_i,\omega_j\right).\]
\end{definition}
\subsection{The Plancherel theorem of SQFT}\label{su4.2}

In Section \ref{su3.2}, we extended $\ {{\mathcal{F}}_r|}_{L^1\cap L^2}\ $to $L^2\left(G^2, \mathbb{H}\right)$. The RQFT on $L^2\left(G^2, \mathbb{H}\right)\ $has more symmetry than RQFT in $L^1\left(G^2, \mathbb{H}\right)$. The relation ${\mathcal{F}}_sf={\mathcal{F}}_r(\mathcal{W}f)$ drives us to extend ${{\mathcal{F}}_s|}_{L^1\cap L^2}$ to $L^2\left(G^2, \mathbb{H}\right)$.
\begin{definition}\label{de4.6}
For every $f\in L^2\left(G^2, \mathbb{H}\right)$, the SQFT of $\Omega_sf$ is defined by
\begin{equation} \label{eq4.3}
\Omega_sf\omega :=\Omega_r(\mathcal{W}f).
\end{equation}
In fact, we can define $\Omega_s$ starting from the original definition of ${\mathcal{F}}_s$ and then taking $L^2$- norm limit. Equation \eqref{eq4.3} gives us a different but actually equivalent form of $\Omega_s.$
\end{definition}
\begin{theorem}\label{th4.7}
Suppose that $f\in L^2\left(G^2, \mathbb{H}\right)$, and $g\in L^2\left(\widehat{G^2}, \mathbb{H}\right)$. Then the following assertions hold:
\begin{enumerate}
\item[(i)]
$ {\rm The\ SQFT}\ \Omega_sf$ defined by \eqref{eq4.3} is equal to the $L^2$-limit of the sequence ${\{{\mathcal{F}}_sf_l\}}_l$, where ${\{f_l\}}_l$ is any sequence in$\ L^1\left(G^2, \mathbb{H}\right)\cap L^2\left(G^2, \mathbb{H}\right)$ converging to $f$ in the $L^2$-norm. If $f\in L^1\left(G^2, \mathbb{H}\right)\cap L^2\left(G^2, \mathbb{H}\right)$, then $\Omega_sf={\mathcal{F}}_sf$.
\item[(ii)]
The transform $\Omega_s$is a bijection on $L^2\left(\widehat{G^2}, \mathbb{H}\right)$ and ${\Omega_s}^{-1}g={\mathcal{W}}^{-1}{\Omega_r}^{-1}g$. Furthermore,$\ {\Omega_s}^{-1}g$ is equal to the $L^2-$limit of the sequence ${\{{{\mathcal{F}}_q}^{-1}g_l\}}_l$, where ${\{g_l\}}_l$ is any sequence in$\ L^1\left(\widehat{G^2}, \mathbb{H}\right)\cap L^2\left(\widehat{G^2}, \mathbb{H}\right)$ converging to$\ g$ in the $L^2$-norm. If $g\in L^1\left(\widehat{G^2}, \mathbb{H}\right)\cap L^2\left(\widehat{G^2}, \mathbb{H}\right)$, then
\[{\Omega_s}^{-1}g={{\mathcal{F}}_s}^{-1}g.\]
\end{enumerate}
\end{theorem}
\begin{proof}
The part (i) is a consequence of \eqref{eq4.3} and the definition of $\Omega_r$. The part (ii) is a consequence of \eqref{eq4.3} and Theorem \ref{th3.12}.
\end{proof}
As an immediate consequence of Theorems \ref{thkh3.10} and \ref{th3.12}, we present the following result.
\begin{theorem}\label{th4.8}
If $f,g\in L^2\left(G^2, \mathbb{H}\right)$, then $\left\langle \Omega_sf,g\right\rangle =\left\langle \mathcal{W}f,{\Omega_r}^{-1}g\right\rangle $.
\end{theorem}
Having defined the SQFT for functions in $L^2\left(\widehat{G^2}, \mathbb{H}\right)$, we obtain the following Parseval's identity.

\begin{theorem}\label{th4.9}
Suppose that $f,g\in L^2\left(G^2, \mathbb{H}\right)$, $F_s=\Omega_sf,$ and ${{\rm \ }\Gamma }_s=\Omega_sg$; then
\[{\left\|F_s\right\|}_2={\left\|f\right\|}_2.\]
Furthermore if
\[{{\rm p}}_0+i{{\rm p}}_1+j{{\rm p}}_2+k{{\rm p}}_3=\int_{G^2}{f\left(x_1,x_2\right)}\overline{g\left(x_1,x_2\right)}{d^2}_{{\mu }_{G^2}}\left(x_1,x_2\right)\]
and
\[{{\rm q}}_0+i{{\rm q}}_1+j{{\rm q}}_2+k{{\rm q}}_3=\int_{\widehat{G^2}}{{\mathcal{F}}_sf\left(\omega_i,\omega_j\right)}\overline{{{\rm \ }\Gamma }_s\left(\omega_i,\omega_j\right)}{d^2}_{{\mu }_{\widehat{G^2}}}\left(\omega_i,\omega_j\right), \]
then ${{\rm p}}_m={{\rm q}}_m$, (\textit{m} = 0,1). Moreover, if both $f$ and g are ${\mathbb{C}}_i$-valued or even with respect to first variable, then ${{\rm p}}_m={{\rm q}}_m (\textit{m} = 0,1,2,3)$.
\end{theorem}
\begin{proof}
Firstly, we show that Parseval's identity of SQFT holds. Applying Parseval's identity of RQFT and Proposition \ref{pr4.3} (iii), we have
\[{\left\|F_s\right\|}^2_2=\left\langle \Omega_sf,\Omega_sf\right\rangle =\left\langle \Omega_s(\mathcal{W}f),\Omega_s(\mathcal{W}f)\right\rangle =\left\langle \mathcal{W}f,\mathcal{W}f\right\rangle ={\left\|\mathcal{W}f\right\|}^2_2={\left\|f\right\|}^2_2.\]
By using Parseval's identity of SQFT to ${\left\|f+g\right\|}^2_2={\left\|F_s+{{\rm \ }\Gamma }_s\right\|}^2_2,\ {\left\|f+ig\right\|}^2_2={\left\|F_s+{{\rm i\ }\Gamma }_s\right\|}^2_2\ $respectively, we get ${{\rm p}}_m={{\rm q}}_m (m = 0,1)$. If both $f$ and g are ${\mathbb{C}}_i-$valued or even with respect to first variable, then $\Omega_sf=\Omega_rf$ and $\Omega_sg=\Omega_rg$; therefore

$\left\langle \Omega_sf,\Omega_sg\right\rangle =\left\langle \Omega_rf,\Omega_rg\right\rangle =\left\langle f,g\right\rangle $, that is, ${{\rm p}}_m={{\rm q}}_m (m= 0,1,2,3)$.
\end{proof}
\section{ Discussions and Conclusions}\label{se5}

Due to the non-commutatively of multiplication of quaternions, there are at least eight types of QFTs and we only consider two typical types of them. How about the rest of QFTs?
\begin{enumerate}
\item[1.]
The left-sided QFT (LQFT) $\overline{\omega_i\left(x_1\right)}\overline{\omega_j\left(x_2\right)}f(.,.)$ follows a similar pattern to RQFT, with the kernel moves to the left side. As left-sided QFT is right H-linear, we only need to revise the definition of inner product in $L^2\left(G^2, \mathbb{H}\right)$ to be ${\left\langle f,g\right\rangle }_{L^2\left(G^2, \mathbb{H}\right)}=\ \int_{G^2}{\overline{f\left(x_1,x_2\right)}g\left(x_1,x_2\right){d^2}_{{\mu }_{G^2}}\left(x_1,x_2\right)}$.
Then the results of RQFT still hold for LQFT case.
\item[2.]
If $i\ {\rm and}\ j$ are substituted into ${\mu }_1{\rm and}\ {\mu }_2,\ \ $respectively, where ${\mu }_1\ {\rm and}\ {\mu }_2$ are any two perpendicular unit pure imaginary quaternions, all of above results still hold.
\item[3.]
For a locally compact abelian group $G$, we may consider $L^2\left(G, \mathbb{C}\right)$ as a $\mathbb{C}$-subspace of $L^2\left(G^2, \mathbb{H}\right)$, by the mapping $f\longmapsto f_\mathbb{H}$, where $f_\mathbb{H}(x,y)=f(x),\  f\in L^2\left(G, \mathbb{C}\right)$. If $\hat{f}$  is the classical Fourier transform of $f$, then $\hat{f}(\omega )={\mathcal{F}}_r(f_\mathbb{H})(\omega,1)$. Therefore based on the proof of inverse right quaternionic Fourier transform, we may prove the classical inverse Fourier transform, Plancherel theorem and other properties. Our technique is different,  and more visible from the classical one.
\end{enumerate}

\end{document}